\documentclass[11pt]{amsart}
\usepackage{amssymb,amscd}
\usepackage[arrow,matrix,curve]{xy}

\newcommand\CO{\mathbb{C}}
\newcommand\RE{\mathbb{R}}
\newcommand\ZA{\mathbb{Z}}

\newcommand\TT{\mathbb{T}}

\newcommand\Tr{\operatorname{Tr}}
\newcommand\Det{\operatorname{Det}}
\newcommand\Hom{\operatorname{Hom}}
\newcommand\rk{\operatorname{rk}}
\newcommand\res{\operatorname{res}}
\newcommand\spec{\operatorname{spec}}
\newcommand\GL{\operatorname{GL}}

\newcommand\vol{\operatorname{vol}}
\newcommand\im{\operatorname{im}}
\newcommand\cpct{_{\mathrm c}}
\newcommand\E{\mathcal E}

\newcommand\odd{^{\mathrm{odd}}}
\newcommand\even{^{\mathrm{even}}}
\newcommand\I{_{\mathrm{I}}}
\newcommand\II{_{\mathrm{II}}}
\def\GL{\operatorname{GL}}
\newcommand{\nno}{\nonumber \\}
\theoremstyle{plain}
\newtheorem{theorem}{Theorem}[section]
\newtheorem{lemma}[theorem]{Lemma}
\newtheorem{proposition}[theorem]{Proposition}
\newtheorem{corollary}[theorem]{Corollary}

\theoremstyle{definition}

\newtheorem{definition*}{Definition}

\theoremstyle{remark}

\newtheorem{remarks*}{Remarks}

\begin{document}

  \begin{flushright}
{\tt arXiv:0810.4204v6[math.DG]}\\
revised: March, 2011
  \end{flushright}

\vspace{1cm}

\title[Analytic torsion for twisted de Rham complexes]
{Analytic torsion for twisted de Rham complexes}

\author{Varghese Mathai}
\address{Department of Mathematics, University of Adelaide,
Adelaide 5005, Australia}
\email{mathai.varghese@adelaide.edu.au}

\author{Siye Wu}
\address{Department of Mathematics, University of Colorado,
Boulder, Colorado 80309-0395, USA and
Department of Mathematics, University of Hong Kong, Pokfulam Road, 
Hong Kong, China} 
\email{swu@maths.hku.hk}

\begin{abstract}
We define analytic torsion $\tau(X,\E,H)\in\det H^\bullet(X,\E,H)$ for the
twisted de Rham complex, consisting of the spaces of differential forms on
a compact oriented Riemannian manifold $X$ valued in a flat vector bundle
$\E$, with a differential given by $\nabla^\E+H\wedge\,\cdot\,$, where 
$\nabla^\E$ is a flat connection on $\E$, $H$ is an odd-degree closed
differential form on $X$, and $H^\bullet(X,\E,H)$ denotes the cohomology
of this $\ZA_2$-graded complex.
The definition uses pseudodifferential operators and residue traces.
We show that when $\dim X$ is odd, $\tau(X,\E,H)$ is independent of the
choice of metrics on $X$ and $\E$ and of the representative $H$ in the
cohomology class $[H]$.
We define twisted analytic torsion in the context of generalized geometry
and show that when $H$ is a $3$-form, the deformation $H\mapsto H-dB$, where
$B$ is a $2$-form on $X$, is equivalent to deforming a usual metric $g$ to
a generalized metric $(g,B)$.
We demonstrate some basic functorial properties.
When $H$ is a top-degree form, we compute the torsion, define its simplicial
counterpart and prove an analogue of the Cheeger-M\"uller Theorem.
We also study the twisted analytic torsion for $T$-dual circle bundles 
with integral $3$-form fluxes.\\
\end{abstract}



\maketitle

\section*{Introduction}

Let $X$ be a compact oriented smooth manifold (without boundary) and 
$\rho\colon\pi_1(X)\to\GL(E)$, an orthogonal or unitary representation
of the fundamental group $\pi_1(X)$ on a vector space $E$.
The Reidemeister-Franz torsion, or $R$-torsion, of $\rho$ is defined in
terms of a triangulation of $X$.
In \cite{RS,RS2}, Ray and Singer introduced its analytic counterpart, 
an alternating product of the regularized determinants of Laplacians, 
and conjectured that the latter is equal to the $R$-torsion.
(For lens spaces, the equality of the two torsions was established in
\cite{Ray}.)
The Ray-Singer conjecture was proved independently by Cheeger \cite{C79}
and M\"uller \cite{M78} for orthogonal or unitary representations of the
fundamental group and was extended to unimodular representations by
M\"uller \cite{M93}.
Another proof of the Cheeger-M\"uller theorem, as well as an extension of 
it to arbitrary flat bundles, is due to Bismut and Zhang \cite{BZ}, who 
used the Witten deformation technique.

In this paper we generalize the classical construction of the Ray-Singer
torsion to the twisted de Rham complex with an odd-degree differential 
form as flux and with coefficients in a flat vector bundle.
The twisted de Rham complex was first defined for $3$-form fluxes by Rohm and
Witten in the appendix of \cite{RW} and has played an important role in string
theory \cite{BCMMS,AS}, for the Ramond-Ramond fields (and their charges) in
type II string theories lie in the twisted cohomology of spacetime.
$T$-duality of the type II string theories on compactified spacetime gives
rise to a duality isomorphism of twisted cohomology groups \cite{BEM}.

Let $H$ be a closed differential form of odd degree on $X$ and 
$\rho\colon\pi_1(X)\to\GL(E)$, a representation of $\pi_1(X)$ 
on a finite dimensional vector space $E$.
Denote by $\E$ the corresponding flat bundle over $X$ with the 
canonical flat connection $\nabla^\E$.
The twisted de Rham complex is the $\ZA_2$-graded complex
$(\Omega^\bullet(X,\E),\nabla^\E+H\wedge\,\cdot\,)$.
Its cohomology, denoted by $H^\bullet(X,\E,H)$, is called the twisted
de Rham cohomology. 
We show that the twisted cohomology groups are invariant under scalings
of $H$ provided its degree is at least $3$ and under smooth homotopy
equivalences that match the cohomology classes of the flux forms.
We establish Poincar\'e duality and K\"unneth isomorphism for twisted
cohomology groups.
We define analytic torsion of the twisted de Rham complex 
$\tau(X,\E,H)\in\det H^\bullet(X,\E,H)$ as a ratio of zeta-function
regularized determinants of partial Laplacians, multiplied by the
ratio of volume elements of the cohomology groups.
While the de Rham complex has a $\ZA$-grading, the twisted de Rham 
complex is only $\ZA_2$-graded.
As a result, analytic techniques used to establish the basic properties
in the classical case have to be generalized accordingly.
These regularized determinants turn out to be more complicated to define,
as they require properties of pseudodifferential projections.
The definition is based on the fact that the non-commutative residues or
the Guillemin-Wodzicki residue traces \cite{Wo,Gui} of such projections
vanish \cite{Wo,BrLe,Gr03}.
We show that when $\dim X$ is odd, $\tau(X,\E,H)$ is independent of the
choice of the Riemannian metric on $X$ and the Hermitian metric on $\E$.
The torsion $\tau(X,\E,H)$ is also invariant (under a natural identification)
if $H$ is deformed within its cohomology class.
The comparison of the deformations of the metrics and of the flux leads
naturally to the concept of generalized metric \cite{Gua2}.
We define twisted analytic torsion in the context of generalized geometry
and show that when $H$ is a $3$-form, the deformation $H\mapsto H-dB$, where
$B$ is a $2$-form on $X$, is equivalent to deforming a usual metric $g$ to
a generalized metric $(g,B)$.
We establish some basic functorial properties of this torsion.
We then compute the torsion for odd-dimensional manifolds with a 
top-degree flux form, which is especially useful for $3$-manifolds.
When the degree of $H$ is sufficiently high we introduce a combinatorial
counterpart of $\tau(X,\E,H)$ and show that they are equal when $H$ is
a top-degree form.
Finally, if $(X,H)$ and $(\widehat X,\widehat H)$ are $T$-dual circle 
bundles with background fluxes, then the $T$-duality isomorphism identifies
the determinant lines 
$\det H^\bullet(X,H)\cong(\det H^\bullet(\widehat X,\widehat H))^{-1}$.
Under this identification, we relate the twisted torsions for $3$-dimensional
$T$-dual circle bundles with integral $3$-form fluxes.

The outline of the paper is as follows.
In \S1, we set up the notation in the paper and review the twisted de Rham
complex and its cohomology \cite{RW,BCMMS} with an odd-degree closed
differential form as flux and with coefficients in a flat vector bundle
associated to a representation of the fundamental group.
In \S2, we introduce the key definition of the analytic torsion of the 
twisted de Rham complex using the vanishing of non-commutative residues
of pseudodifferential projections.
In \S3, we show that the twisted analytic torsion is independent of the
metrics on the manifold and on the flat bundle.
We also show that it depends on the flux only through its cohomology class.
The relation to generalized geometry is then explored.
In \S4, we establish the basic functorial properties of the analytic torsion
for the twisted de Rham complex.
\S5 contains calculations of analytic torsion for the twisted de Rham complex
and a simplicial version of it under certain restrictions.
In this special case, the analogue of the Cheeger-M\"uller theorem is
established.
Finally, we study the behavior of the twisted analytic torsion under
$T$-duality for circle bundles with a closed $3$-form as flux. 

There is an extensive literature on the torsion of $\ZA$-graded complexes.
Analytic torsion has been studied for manifolds with boundary
\cite{LR,Lv,Vi,DF,BrMa}, for the Dolbeault complex \cite{RS3,BGS,BL},
in the equivariant setting \cite{LR, Lv, BZ2, Bu, BG}, and for fibrations
\cite{DM,LST,Ma2}, where torsion forms \cite{BGS,BK,BLo,Ma,Ma2} appear.
The analytic torsion was also identified as the partition function of certain
topological field theories and was studied for arbitrary ($\ZA$-graded)
elliptic complexes \cite{Sch}. 
Recently, refined and complex-valued analytic torsions have been introduced
and studied \cite{BK1,BK2,BH,SZ}.
It is tempting to extend these developments to $\ZA_2$-graded complexes.
Until recently, there appears to be no simplicial analogue of the twisted de
Rham complex, except in a special case in \S\ref{comb}, since the cup product
is in general not graded commutative at the level of the cochain complex.

\medskip

\noindent {\bf Acknowledgments}
V.M.~acknowledges support from the Australian Research Council.
S.W.~is supported in part by CERG HKU705407P.
We thank M.~Braverman, G.~Grubb, X.~Ma and W.~Zhang for discussions and
the referees for comments and suggestions. 

\section{Twisted de Rham complexes}
To set up the notation in the paper, we review the twisted de Rham
cohomology \cite{RW,BCMMS} with an odd-degree flux form and with
coefficients in a flat vector bundle.
We show that the twisted cohomology does not change under the scalings of
the flux form when its degree is at least $3$.
We also establish homotopy invariance, Poincar\'e duality and K\"unneth
isomorphism for these cohomology groups.

\subsection{Flat vector bundles, representations and Hermitian metrics}
\label{sect:scalar}
Let $X$ be a connected, compact, oriented smooth manifold.
Let $\rho\colon\pi_1(X)\to\GL(E)$ be a representation of the fundamental
group $\pi_1(X)$ on a vector space $E$.
The associated vector bundle $p\colon\E\to X$ is given by
$\E=(E\times\widetilde X)/\sim$, where $\widetilde X$ denotes the universal
covering of $X$ and $(v,x\gamma)\sim(\rho(\gamma)v,x)$ for all
$\gamma\in\pi_1(X)$, $x\in\widetilde X$ and $v\in E$.
If the representation $\rho$ is real or complex, so is the bundle $\E$,
respectively.
A smooth section $s$ of $\E$ can be uniquely represented by a smooth 
equivariant map $\phi\colon\widetilde X\to E$, satisfying
$\phi(x\gamma)=\rho(\gamma)^{-1}\phi(x)$ for all $\gamma\in\pi_1(X)$
and $x\in\widetilde X$. 

Given any vector bundle $p\colon\E\to X$ over $X$, denote by $\Omega^i(X,\E)$
the space of smooth differential $i$-forms on $X$ with values in $\E$. 
A {\it flat connection} on $\E$ is a linear map
$$
\nabla^\E\colon\Omega^i(X,\E)\to\Omega^{i+1}(X,\E)
$$
such that 
$$
\nabla^\E(f\omega)=df\wedge\omega+f\,\nabla^\E\omega 
\qquad\text{and}\qquad(\nabla^\E)^2=0 
$$
for any smooth function $f$ on $X$ and any $\omega\in\Omega^i(X,\E)$.
If the vector bundle $\E$ is associated with a representation $\rho$
as in the previous paragraph, an element of $\Omega^\bullet(X,\E)$
can be uniquely represented as a $\pi_1(X)$-invariant element in 
$E\otimes\Omega^\bullet(\widetilde X)$.
If $\omega\in\Omega^\bullet(\widetilde X)$ and $v\in E$,
then $v\otimes\omega$ is said to be $\pi_1(X)$-invariant
if $\rho(\gamma)v\otimes\gamma^*\omega=v\otimes\omega$
for all $\gamma\in\pi_1(X)$.
On such a vector bundle, there is a {\it canonical flat connection}
$\nabla^\E$ given by, under the above identification,
$\nabla^\E(v\otimes\omega)=v\otimes d\omega$, 
where $d$ is the exterior derivative on forms.

The usual wedge product on differential forms can be extended to
$$
\wedge\colon\Omega^i(X)\otimes\Omega^{j}(X,\E)\to\Omega^{i+j}(X,\E).
$$
Together with the evaluation map $\E\otimes\E^*\to\CO$, we have another product
$$
\wedge\colon\Omega^i(X,\E)\otimes\Omega^{j}(X,\E^*)\to\Omega^{i+j}(X).
$$
A Riemannian metric $g_X$ defines the Hodge star operator 
$$
\ast\colon\Omega^i(X,\E)\to\Omega^{n-i}(X,\E),
$$
where $n=\dim X$.
A Euclidean or Hermitian metric $g_\E$ on $\E$ determines an $\RE$-linear
bundle isomorphism $\sharp\colon\E\to\E^*$, which extends to an $\RE$-linear
isomorphism 
$$
\sharp\colon\Omega^i(X,\E)\to\Omega^i(X,\E^*);
$$
when $\E$ is complex, the isomorphism is conjugate linear. 
One sets $\Gamma=\ast\;\sharp=\sharp\;\ast$ and for any 
$\omega,\omega'\in\Omega^i(X,\E)$, let
$$
(\omega,\omega')=\int_X\omega\wedge\Gamma\,\omega'.
$$
This makes each $\Omega^i(X,\E)$, $0\le i\le n$, a pre-Hilbert space.

When $\E$ is associated to an orthogonal or unitary representation $\rho$
of $\pi_1(X)$, $g_\E$ can be chosen to be compatible with the canonical
flat connection. 
This is not possible in general.
We will not assume that $\rho$ is unimodular except in \S\ref{sect:calc},
where we calculate the torsion and establish an simplicial analogue under
special conditions.

\subsection{Twisted de Rham cohomology}\label{sect:twistedDR} 
Given a flat vector bundle $p\colon\E\to X$ and an odd-degree,
closed differential form $H$ on $X$, we set 
$\Omega^{\bar 0}(X,\E):=\Omega\even(X,\E)$,
$\Omega^{\bar 1}(X,\E):=\Omega\odd(X,\E)$
and $\nabla^{\E,H}:=\nabla^\E+H\wedge\,\cdot\;$.
We are primarily interested in the case when $H$ does not contain a $1$-form
component, which can be absorbed in the flat connection $\nabla^\E$.
We define the {\it twisted de Rham cohomology groups of $\E$} as the quotients
$$
H^{\bar k}(X,\E,H)=
\frac{\ker\,(\nabla^{\E,H}\colon\Omega^{\bar k}(X,\E)\to
 \Omega^{\overline{k+1}}(X,\E))}
{\im\,(\nabla^{\E,H}\colon\Omega^{\overline{k+1}}(X,\E)\to
 \Omega^{\bar k}(X,\E))},\quad k=0,1.
$$
Here and below, the bar over an integer means taking the value modulo $2$.
The groups $H^{\bar k}(X,\E,H)$ ($k=0,1$) are manifestly independent of the
choice of the Riemannian metric on $X$ or the Hermitian metric on $\E$.
The corresponding {\em twisted Betti numbers} are denoted by 
$$
b_{\bar k}=b_{\bar k}(X,\E,H):=\dim H^{\bar k}(X,\E,H),\quad k=0,1.
$$
Suppose $H$ is replaced by $H'=H-dB$ for some $B\in\Omega^{\bar0}(X)$,
then there is an isomorphism $\varepsilon_B:=e^B\wedge\cdot\,\colon
\Omega^\bullet(X,\E)\to\Omega^\bullet(X,\E)$ satisfying
$$
\varepsilon_B\circ\nabla^{\E,H}=\nabla^{\E,H'}\circ\varepsilon_B.
$$
Therefore the Poincar\'e lemma holds for the twisted differential when the
space is contractible. 
In general, $\varepsilon_B$ induces an isomorphism (denoted by the same)
\begin{equation}\label{e^B}
\varepsilon_B\colon H^\bullet(X,\E,H)\to H^\bullet(X,\E,H')
\end{equation}
on the twisted de Rham cohomology.
So the twisted Betti numbers depend only on the de Rham cohomology class
of $H$. 
If they are finite, the Euler characteristic
$$
\chi(X,\E,H):=\sum_{k=0,1}(-1)^kb_{\bar k}(X,\E,H)=\chi(X,\E)=\chi(X)\rk\E
$$
is independent of $H$ and depends on $\E$ only through its rank.
If $X$ is odd-dimensional, then $\chi(X,\E,H)=\chi(X,\E)=\chi(X)=0$.

When $H$ is a $1$-form, $H^\bullet(X,\E,H)$ has a $\ZA$-grading but the
dimension can jump when $H$ is rescaled by a non-zero number \cite{Nov,Paz}.
The behavior is qualitatively different when the degree of $H$ is at least $3$.

\begin{proposition}\label{scaling}
Let $\E$ be a flat vector bundle over $X$ and $H$, an odd-degree closed
form on $X$.
Suppose $H=\sum_{i\ge1}H_{2i+1}$, where each $H_{2i+1}$ is a $(2i+1)$-form.
For any $\lambda\in\RE$, let $H^{(\lambda)}=\sum_{i\ge1}\lambda^i H_{2i+1}$.
Then $H^\bullet(X,\E,H)\cong H^\bullet(X,\E,H^{(\lambda)})$ if $\lambda\ne0$.
\end{proposition}

\begin{proof}
For any $\lambda$, let $c_\lambda$ act on $\Omega^\bullet(X,\E)$ by 
multiplying $\lambda^{[\frac{i}{2}]}$ on $i$-forms.
Then $H^{(\lambda)}=c_\lambda(H)$ and
$c_\lambda\circ\nabla^{\E,H}=\lambda^k\;\nabla^{\E,H^{(\lambda)}}\circ
c_\lambda$ on $\Omega^{\bar k}(X,\E)$ for $k=0,1$.
If $\lambda\ne0$, then $c_\lambda$ induces the desired isomorphism on
twisted cohomology groups.
\end{proof}

Although the twisted differential $\nabla^{\E,H}$ does not preserve the
$\ZA$-grading of the de Rham complex, it does respect a filtration $F$
given by \cite{RW,AS}
$$
F^p\Omega^{\bar k}(X,\E)=\bigoplus_{i\ge p,\;i=k\!\!\!\!\mod 2}\Omega^i(X,\E).
$$
This filtration gives rise to a spectral sequence $\{E_r^{pq},\delta_r\}$
converging to the twisted cohomology $H^\bullet(X,\E,H)$.
Without loss of generality, we assume that $H$ contains no component of
$1$-form, which can be absorbed in the flat connection.
That is, $H=H_3+H_5+\cdots$, where $H_i$ is an $i$-form ($i=3,5,\cdots$).
Then
$$
E_2^{p\bar q}=\left\{\begin{array}{ll}
H^p(X,\E), & \mathrm{if}\;\; q=0,\\
0, &  \mathrm{if}\;\; q=1.
\end{array}\right.
$$
As usual, $E_{r+1}^\bullet$ is computed from a complex
$(E_r^\bullet,\delta_r)$ for $r\ge2$.
We have $\delta_2=\delta_4=\cdots=0$, while $\delta_3,\delta_5,\cdots$
are given by the cup products with $[H_3],[H_5],\cdots$ and by the
higher Massey products with them \cite{RW,AS}.
Proposition~\ref{scaling} can also be derived by using this spectral sequence.

\subsection{Homotopy invariance of twisted de Rham cohomology}
Given $X$, $\E$ and $H$ as above, any smooth map $f\colon Y\to X$
(where $Y$ is another smooth manifold) induces a homomorphism 
$$
f^*\colon H^\bullet(X,\E,H)\to H^\bullet(Y,f^*\E,f^*H).
$$
We will show that this map depends only on the homotopy class of $f$.
For simplicity, we assume that $\E$ is a trivial line bundle with the
trivial connection and denote $\nabla^{\E,H}$ by $d^H$ in this case. 
Let $I$ be the unit interval.

\begin{lemma}\label{lem:homotopy}
Let $\pi\colon X\times I\to X$ denote the projection onto $X$ and
$s\colon X\to X\times I$, the map $x\mapsto(x,0)$ for $x\in X$.
Then the maps $\pi^*\colon H^\bullet(X,H)\to H^\bullet(X\times I,\pi^*H)$
and $s^*\colon H^\bullet(X\times I,\pi^*H)\to H^\bullet(X,H)$ are
inverses to each other.
\end{lemma}
 
\begin{proof}
Clearly, $s^*\circ\pi^*$ is the identity map on $H^\bullet(X,H)$.
By the homotopy invariance of de Rham cohomology, there is a chain homotopy
operator $K\colon\Omega^i(X\times I)\to\Omega^{i-1}(X\times I)$ defined by
(cf.~\S I.4 of \cite{BoTu}) $K(\pi^*\alpha\,f(x,t))=0$ and
$K(\pi^*\alpha\wedge f(x,t)dt)=(-1)^i\,\pi^*\alpha\int_0^tf(x,t')dt'$, where
$\alpha\in\Omega^i(X)$, $f\in C^\infty(X\times I)$ and $x\in X$, $t\in I$.
For any $\omega\in\Omega^\bullet(X\times I)$, we have
$$
\omega-\pi^*s^*\omega=dK\omega+Kd\omega.
$$
Since $K(\pi^*H\wedge\omega)=-\pi^*H\wedge K(\omega)$, we have
$$
\omega-\pi^*s^*\omega=d^{\pi^*H}K\omega+Kd^{\pi^*H}\omega.
$$
Therefore $\pi^*\circ s^*$ is the identity map on
$H^\bullet(X\times I,\pi^*H)$.
\end{proof}

\begin{proposition}\label{hom-inv}
Let $f_0,f_1\colon Y\to X$ be two smooth maps that are homotopic.
Then there exists $B\in\Omega^{\bar0}(Y)$ such that $f_1^*H=f_0^*H-dB$
and the following diagram commutes
$$
\xymatrix@=1pc{& H^\bullet(X,H) \ar[dl]_{f_0^*} \ar[dr]^{f_1^*} &\\
H^\bullet(Y,f_0^*H) \ar[rr]^{\varepsilon_B} && H^\bullet(Y,f_1^*H).}
$$
\end{proposition}
 
\begin{proof}
Let $\pi\colon Y\times I\to Y$ be the projection onto $Y$.
Define the smooth maps $s_j\colon Y\to Y\times I$ ($j=0,1$) by
$s_j(y)=(y,j)$, where $y\in Y$.
Then a homotopy between $f_0$ and $f_1$ is a smooth map
$F\colon Y\times I\to X$ such that $f_j=F\circ s_j$ for $j=0,1$.
There exists $\tilde B\in\Omega^{\bar0}(Y\times I)$ such that 
$F^*H=\pi^*f_0^*H-d\tilde B$ and $s_0^*\tilde B=0$.
Let $B=s_1^*\tilde B\in\Omega^{\bar0}(Y)$ and $\tilde B'=\tilde B-\pi^*B$.
Then $f_1^*H-f_0^*H=-dB$, $F^*H=\pi^*f_1^*H-d\tilde B'$ and $s_1^*\tilde B'=0$.
There is a commutative diagram
$$
\xymatrix@=1pc{& & & H^\bullet(Y\times I,f_0^*H)\ar[dd]^{\varepsilon_{\pi^*B}}
 \ar[dl]_{\varepsilon_{\tilde B}}
 & & H^\bullet(Y,f_0^*H) \ar[ll]_{{}\quad\;\;\pi^*} \ar[dd]^{\varepsilon_B} \\
H^\bullet(X,H)\ar[rr]^-{F^*} & & H^\bullet(Y\times I,F^*H) \\
& & & H^\bullet(Y\times I,f_1^*H)\ar[ul]^{\varepsilon_{\tilde B'}} 
& & H^\bullet(Y,f_1^*H).\ar[ll]_{{}\quad\;\;\pi^*}}
$$
By Lemma~\ref{lem:homotopy}, $s_0=(\pi^*)^{-1}\colon
H^\bullet(Y\times I,\pi^*f_0^*H)\to H^\bullet(Y,f_0^*H)$.
Since $s_0^*\tilde B=0$, $(\pi^*)^{-1}\circ\varepsilon_{\tilde B}^{-1}
=s_0^*\colon H^\bullet(Y\times I,F^*H)\to H^\bullet(Y,f_0^*H)$.
Similarly, $(\pi^*)^{-1}\circ\varepsilon_{\tilde B'}^{-1}=s_1^*
\colon H^\bullet(Y\times I,F^*H)\to H^\bullet(Y,f_1^*H)$.
The result follows since $f_j^*= s_j^*\circ F^*$, $j=0,1$.
\end{proof}

It is clear from the proof that in addition to $f_1^*H=f_0^*H-dB$,
$B$ has to come from the homotopy.
In fact, $B=-K(F^*H)$, where $K$ is the homotopy chain map in the proof
of Lemma~\ref{lem:homotopy}.
If $H$ is fixed in the homotopy process, then $f_0^*=f_1^*$.

\begin{corollary}\label{cor:hmtp}
Suppose $X$, $X'$ are smooth manifolds and $H$, $H'$ are closed, odd-degree
forms on $X$, $X'$, respectively.
If there is a smooth homotopy equivalence $f\colon X\to X'$ such that
$[f^*H']=[H]$, then $H^\bullet(X,H)\cong H^\bullet(X',H')$.
\end{corollary}

\subsection{Products and Poincar\'e duality}
Unlike the de Rham cohomology group $H^\bullet(X)$, the twisted de Rham
cohomology group $H^\bullet(X,H)$ is not a ring for a fixed $H\ne0$.
Instead, the flux $H$ adds in the cup product.

\begin{lemma}\label{lemma:cup}
Let $\E,\E'$ be flat vector bundles and let $H,H'$ be closed odd-degree
differential forms on a smooth manifold $X$.
Then the wedge product
\[   \Omega^{\bar k}(X,\E)\otimes\Omega^{\bar l}(X,\E') 
\stackrel{\wedge}{\longrightarrow}\Omega^{\overline{k+l}}(X,\E\otimes\E')  \]
induces a natural cup product
\[   H^{\bar k}(X,\E,H)\otimes H^{\bar l}(X,\E',H')
  \stackrel{\cup}{\longrightarrow}H^{\overline{k+l}}(X,\E\otimes\E',H+H'), \]
where $k,l=0,1$.
In particular, there is a natural cup product
\[   H^{\bar k}(X,\E,H)\otimes H^{\bar l}(X,\E^*,-H)
  \stackrel{\cup}{\longrightarrow}H^{\overline{k+l}}(X).                   \]
\end{lemma}

\begin{proof}
The existence of the cup product follows from the formula
\[  \nabla^{\E\otimes\E'\!,\,H+H'}(\omega\wedge\omega')=\nabla^{\E,H}\omega
    \wedge\omega'+(-1)^{\bar k}\,\omega\wedge\nabla^{\E'\!,\,H'}\omega',   \]
where $\omega\in\Omega^{\bar k}(X,\E)$, $\omega'\in\Omega^{\bar l}(X,\E')$.
When $\E'=\E^*$, $H'=-H$, the cup product, composed with the pairing between
$\E$ and $\E^*$, takes values in $H^{\overline{k+l}}(X)$.
\end{proof}

It is possible to define, following \S\ref{sect:twistedDR}, the twisted de Rham
cohomology groups of compact support, denoted by $H\cpct^{\bar k}(X,\E,H)$, by
using the space of differential forms $\Omega\cpct^{\bar k}(X,\E)$ of compact
support.
(There is no restriction on the support of $H$.)
As in Lemma~\ref{lemma:cup}, the wedge product
\[   \Omega^{\bar k}(X,\E)\otimes\Omega\cpct^{\bar l}(X,\E') 
\stackrel{\wedge}{\longrightarrow}\Omega\cpct^{\overline{k+l}}(X,\E\otimes\E')
\]
induces cup products
\[   H^{\bar k}(X,\E,H)\otimes H\cpct^{\bar l}(X,\E',H')
\stackrel{\cup}{\longrightarrow}H\cpct^{\overline{k+l}}(X,\E\otimes\E',H+H') \]
and
\[   H^{\bar k}(X,\E,H)\otimes H\cpct^{\bar l}(X,\E^*,-H)
     \stackrel{\cup}{\longrightarrow}H\cpct^{\overline{k+l}}(X).   \]

\begin{proposition}[Poincar\'e duality]\label{prop:pd}
Let $X$ be an oriented manifold of dimension $n$ with a finite good cover.
Let $\E$ be a flat vector bundle on $X$. 
Suppose $H$ is a closed odd-degree differential form on $X$.
Then, for $k=0,1$, there is a natural isomorphism
\[   H^{\bar k}(X,\E,H)\cong(H\cpct^{\overline{n-k}}(X,\E^*,-H))^*.   \]
\end{proposition}

\begin{proof}
Since $X$ is orientable, we have $H\cpct^n(X)\cong\RE$.
There are pairings between $\Omega^{\bar k}(X,\E)$ and
$\Omega\cpct^{\overline{n-k}}(X,\E^*)$ and between $H^{\bar k}(X,\E,H)$ and
$H\cpct^{\overline{n-k}}(X,\E^*,-H)$; the former is given by integration on
$X$ and the latter is given by the cup product to $H\cpct^{\bar n}(X)$
followed by a projection to $H\cpct^n(X)$.
Thus we have linear maps
$$\Omega^{\bar k}(X,\E)\to(\Omega\cpct^{\overline{n-k}}(X,\E^*))^*\quad
\mbox{and}\quad H^{\bar k}(X,\E,H)\to(H\cpct^{\overline{n-k}}(X,\E^*,-H))^*.$$
Here $(\Omega\cpct^\bullet(X,\E^*))^*$ is the linear dual of
$\Omega\cpct^\bullet(X,\E^*)$.
To show that the latter is an isomorphism, we observe that it is so if $X$
is contractible, as the two cohomology groups can be calculated explicitly.
For any two open subsets $U,V$ in $X$, we have a morphism of short exact
sequences of $\ZA_2$-graded cochain complexes

$$ \begin{array}{rcccccl}
0\to & \Omega^\bullet(U\cup V,\E) & \to &
\Omega^\bullet(U,\E)\oplus\Omega^\bullet(V,\E)
& \to & \Omega^\bullet(U\cap V,\E) & \to0            \\
& \downarrow & & \downarrow && \downarrow & \\
0\to & \!\!\!\!\!\Omega\cpct^{\bar n-\bullet}(U\cup V,\E^*)^*\!\!\!\!\! 
& \to & \!\!\!\!\!\Omega\cpct^{\bar n-\bullet}(U,\E^*)^*\oplus
\Omega\cpct^{\bar n-\bullet}(V,\E^*)^*\!\!\!\!\! & \to & \!\!\!\!\!
\Omega\cpct^{\bar n-\bullet}(U\cap V,\E^*)^*\!\!\!\!\! & \to0.
\end{array}  $$
This induces a morphism of the Mayer-Vietoris sequences (which are six-term
long exact sequences).
Following \S I.5 of \cite{BoTu}, the result can be proved by an induction on
the number of open sets in the finite good cover and by using the five lemma.
\end{proof}

\begin{proposition}[K\"unneth isomorphism]\label{prop:ku}
For $i=1,2$, let $X_i$ be smooth manifolds.
Suppose $\E_i$ are flat vector bundles over $X_i$ and $H_i$, closed
odd-degree forms on $X_i$, respectively.
Let $\pi_i\colon X_1\times X_2\to X_i$ be the projections.
Set $\E_1\boxtimes\E_2=\pi_1^*\E_1\otimes\pi_2^*\E_2$ and
$H_1\boxplus H_2=\pi_1^*H_1+\pi_2^*H_2$.
If either $X_1$ or $X_2$ has a finite good cover, then, for each $k=0,1$,
there is a natural isomorphism
\[ H^{\bar k}(X_1\times X_2,\E_1\boxtimes\E_2,H_1\boxplus H_2) 
\cong\bigoplus_{l=0,1}H^{\bar l}(X_1,\E_1,H_1)
\otimes H^{\overline{k-l}}(X_2,\E_2,H_2).   \]
\end{proposition}

\begin{proof}
Suppose $X_1$ has a finite good cover.
If $X_1$ is contractible, then $\E_1$ is trivial and $H_1$ is exact on $X_1$.
So are $\pi^*\E_1$ and $\pi_1^*H_1$ on $X_1\times X_2$.
By the isomorphism (\ref{e^B}) and by homotopy invariance
(Corollary~\ref{cor:hmtp}), we get a natural isomorphism
$$ \bigoplus_{l=0,1}H^{\bar k}(X_1,\E_1,H_1)\otimes
    H^{\overline{k-l}}(X_2,\E_2,H_2)\stackrel{\cong}{\longrightarrow}
    H^{\bar k}(X_1\times X_2,\E_1\boxtimes\E_2,H_1\boxplus H_2)   $$
induced by the map $\omega_1\otimes\omega_2\in
\Omega^\bullet(X_1,\E_1)\otimes\Omega^\bullet(X_2,\E_2)\mapsto
\pi_1^*\omega_1\wedge\pi_2^*\omega_2\in\Omega^\bullet(X_1\times X_2,\E_1\boxtimes\E_2)$.
For any two open subsets $U,V$ of $X_1$, there is a morphism of short exact
sequences of $\ZA_2$-graded cochain complexes
{\small
$$ \hspace{-5pt}\begin{array}{rcccccl}
0\to & \!\!\!\!\!\bigoplus\hspace{-15pt}{\displaystyle \mathop{}_{{}_{l=0,1}}}
\Omega^{\bar l}(U\cup V)\otimes\Omega^{\bullet-\bar l}(X_2)\!\!\!\!\! & \to & 
\!\!\!\!\!\bigoplus\hspace{-15pt}{\displaystyle \mathop{}_{{}_{l=0,1}}}
(\Omega^{\bar l}(U)\oplus\Omega^{\bar l}(V))\otimes
\Omega^{\bullet-\bar l}(X_2)\!\!\!\!\! & \to &
\!\!\!\!\!\bigoplus\hspace{-15pt}{\displaystyle \mathop{}_{{}_{l=0,1}}}
\Omega^{\bar l}(U\cap V)\otimes\Omega^{\bullet-\bar l}(X_2)\!\!\!\!\! & \to0 \\
& \downarrow & & \downarrow && \downarrow & \\
0\to & \Omega^\bullet((U\cup V)\times X_2)
& \to & \Omega^\bullet(U\times X_2)\oplus\Omega^\bullet(V\times X_2)
& \to & \Omega^\bullet((U\cap V)\times X_2) & \to0.
\end{array}  $$}\\
Here the obvious dependence on the bundles is suppressed for brevity.
This induces a morphism of the Mayer-Vietoris sequences (which are six-term
long exact sequences).
Following \S I.5 of \cite{BoTu} again, the result can be proved by an
induction on the number of open sets in the finite good cover and by
using the five lemma.
\end{proof}

\section{Analytic torsion of twisted de Rham complexes}

In this section, we define analytic torsion
$\tau(X,\E,H)\in\det H^\bullet(X,\E,H)$ of the twisted de Rham complexes
introduced in \S\ref{sect:twistedDR}.
Since these complexes are only $\ZA_2$-graded, the twisted analytic torsion
is more complicated to define and to study than its classical counterpart.
For simplicity, assume that $X$ is orientable.

\subsection{The construction of analytic torsion}\label{subsect:constr}
To simplify notation, let $C^{\bar k}:=\Omega^{\bar k}(X,\E)$ and let 
$d_{\bar k}=d_{\bar k}^{\E,H}$ be the operator $\nabla^{\E,H}$ acting
on $C^{\bar k}$ ($k=0,1$).
Then $d_{\bar 1}d_{\bar 0}=d_{\bar 0}d_{\bar 1}=0$ and we have a complex
\begin{equation}\label{cplx}
\cdots\stackrel{d_{\bar 1}}{\longrightarrow}C^{\bar 0}
\stackrel{d_{\bar 0}}{\longrightarrow}C^{\bar 1}
\stackrel{d_{\bar 1}}{\longrightarrow}C^{\bar 0}
\stackrel{d_{\bar 0}}{\longrightarrow}\cdots.
\end{equation}
Denote by $d_{\bar k}^\dagger$ the adjoint of $d_{\bar k}$ with respect to
the scalar product of \S\ref{sect:scalar}.
Then the twisted Laplacians
$$
\Delta_{\bar k}=\Delta_{\bar k}^{\E,H}:=d_{\bar k}^\dagger d_{\bar k}
+d_{\overline{k+1}}d_{\overline{k+1}}^\dagger,\quad k=0,1
$$
are elliptic operators and therefore the complex (\ref{cplx}) is elliptic.
By Hodge theory, the natural map 
$\ker(\Delta_{\bar k})\to H^{\bar k}(X,\E,H)$ taking each twisted 
harmonic form to its cohomology class is an isomorphism.
Ellipticity of $\Delta_{\bar k}$ ensures that the twisted Betti numbers 
$b_{\bar k}(X,\E,H)$ ($k=0,1$) are finite.

We establish the relation of Hodge star with adjoint and harmonic form in
the twisted case.

\begin{lemma}\label{lemma:adj}
In the above notations, we have
\[  (d_{\bar k}^{\E,H})^\dagger=(-1)^{k+1}
     \Gamma^{-1}\circ d_{\overline{n+1-k}}^{\E^*,-H}\circ\Gamma,\qquad
     \Gamma\circ\Delta_{\bar k}^{\E,H}=\Delta_{\overline{n-k}}^{\E^*\!,-H}
     \circ\Gamma. \]
\end{lemma}

\begin{proof}
For $\omega\in C^{\bar k}$, $\omega'\in C^{\overline{k+1}}$, we have
\begin{align}
  (d_{\bar k}^{\E,H}\omega,\omega')
  &=\int_X(\nabla^\E\omega+H\wedge\omega)\wedge\Gamma\omega'
  =(-1)^k\int_X\omega\wedge
    (-\nabla^{\E^*}\Gamma\omega'+H\wedge\Gamma\omega')	          \nno
  &=(-1)^{k+1}(\omega,\Gamma^{-1}\,d_{\overline{n+1-k}}^{\E^*\!,-H}\,
       \Gamma\omega').                                    \nonumber
\end{align}
Therefore $(d_{\bar k}^{\E,H})^\dagger=
(-1)^{k+1}\Gamma^{-1}d_{\overline{n+1-k}}^{\E^*\!,-H}\,\Gamma$.
Replacing $\E$ by $\E^*$, $H$ by $-H$ and using the identity
$\Gamma^2=(-1)^{k(n-k)}$ on $C^{\bar k}$, we get  $d_{\bar k}^{\E,H}=
(-1)^{k+1}\Gamma^{-1}(d_{\overline{n+1-k}}^{\E^*\!,-H})^\dagger\,\Gamma$.
Therefore $\Delta_{\bar k}^{\E,H}=
\Gamma^{-1}\Delta_{\overline{n-k}}^{\E^*\!,-H}\,\Gamma$.
\end{proof}

Consequently, the two twisted Laplacians $\Delta_{\bar k}^{\E,H}$ and
$\Delta_{\overline{n-k}}^{\E^*\!,-H}$ have the same spectrum. 
In particular, if $\omega$ is harmonic with respect to $\Delta^{\E,H}$, then
$\Gamma\omega$ is harmonic with respect to $\Delta^{\E^*\!,-H}$.
This provides another proof of Poincar\'e duality (Proposition~\ref{prop:pd})
when $X$ is compact.

The scalar product on $C^{\bar k}$ restricts to one on the space of
twisted harmonic forms $\ker(\Delta_{\bar k})\cong H^{\bar k}(X,\E,H)$.
Let $\{\nu_{\bar k,i}\}_{i=1}^{b_{\bar k}}$ be an oriented orthonormal
basis of $H^{\bar k}(X,\E,H)$ and let $\eta_{\bar k}=\eta_{\bar k}^{\E,H}
:=\nu_{\bar k,1}\wedge\cdots\wedge\nu_{\bar k,b_{\bar k}}$,
the unit volume element.
Then $\eta_{\bar 0}\otimes\eta_{\bar 1}^{-1}\in\det H^\bullet(X,\E,H)$.
The {\em analytic torsion of the twisted de Rham complex} is defined to be
\begin{equation}\label{def-tor}
\tau(X,\E,H):=(\Det'd_{\bar 0}^\dagger d_{\bar 0})^{1/2}
(\Det'd_{\bar 1}^\dagger d_{\bar 1})^{-1/2}
\eta_{\bar 0}\otimes\eta_{\bar 1}^{-1}\in\det H^\bullet(X,\E,H),
\end{equation}
where $\Det'd_{\bar k}^\dagger d_{\bar k}$ denotes the zeta-function
regularized determinant of $d_{\bar k}^\dagger d_{\bar k}$ on the
orthogonal complement of its kernel.
The next subsection is devoted to showing that these determinants make sense.
When $\E$ is the trivial line bundle over $X$ with the trivial connection,
we set $\tau(X,H)=\tau(X,\E,H)$.

We explain the motivation for definition (\ref{def-tor}) by considering
the case $H=0$.
Then $C^{\bar k}=\bigoplus_{i=k\!\!\!\mod 2}C^i$ and
$d_{\bar k}=\sum_{i=k\!\!\!\mod 2}d_i$ ($k=0,1$), where $C^i=\Omega^i(X,\E)$
($0\le i\le n$) and the differentials $d_i=d_i^\E$ ($0\le i\le n-1$) form
the $\ZA$-graded de Rham complex 
$$
0\to C^0\stackrel{d_0}{\longrightarrow}
C^1\stackrel{d_1}{\longrightarrow}\cdots
\stackrel{d_{n-1}}{\longrightarrow}C^n\to0
$$
with $d_id_{i-1}=0$ ($1\le i\le n$).
By spectral theory, $\Det'd_i^\dagger d_i$ ($0\le i\le n-1$) can be defined
and is equal to $\prod_{j=0}^{n-i-1}(\Det'\Delta_{i-j})^{(-1)^j}$, where
$\Delta_i=d_i^\dagger d_i+d_{i-1}d_{i-1}^\dagger$ (with $d_{-1}=d_n=0$) 
is the Laplacian on $C^i$.
Thus the determinant factor in (\ref{def-tor}) is
$$
(\Det'd_{\bar 0}^\dagger d_{\bar 0})^{1/2}
(\Det'd_{\bar 1}^\dagger d_{\bar 1})^{-1/2}
=\prod_{i=0}^{n-1}(\Det'd_i^\dagger d_i)^{(-1)^i/2}
=\prod_{i=0}^n(\Det'\Delta_i)^{(-1)^{i+1}i/2},
$$
which yields the Ray-Singer torsion $\tau(X,\E)$ \cite{RS}.
When $\E$ is the trivial line bundle over $X$, we set $\tau(X)=\tau(X,\E)$.

We wish to point out that the classical signature or Dirac complex, although
being a $2$-term complex, is not of the form (\ref{cplx}) because it does
not satisfy $d_{\bar 1}d_{\bar 0}=d_{\bar 0}d_{\bar 1}=0$.
Therefore, no torsion is defined in these cases.

If $\E$ is a complex vector bundle, the torsion is only defined up to a phase
due to the ambiguity in the choice of the unit volume elements $\eta_{\bar k}$.
Therefore an equality of torsions means that they are equal up to a phase
or that the volume elements can be chosen so that they are equal. 
More intrinsic is the norm on the determinant line (cf.~\cite{Q,BGS,BZ}
for the $\ZA$-graded case) given by
$$
||\cdot||=(\Det'd_{\bar0}^\dagger d_{\bar0})^{1/2}
(\Det'd_{\bar1}^\dagger d_{\bar1})^{-1/2}\,|\cdot|,
$$
where $|\,\cdot\,|$ is the norm induced by the scalar products on
$\ker(\Delta_{\bar k})\cong H^{\bar k}(X,\E,H)$, $k=0,1$.
However, to facilitate the presentation, we will still regard torsions
as (equivalent classes of) elements in the determinant lines. 
Recently, refined and complex-valued analytic torsions were introduced
as well-defined elements of the determinant line \cite{BK1,BK2,BH}.

\subsection{The zeta-function regularized determinants} 
Given a semi-positive definite self-adjoint operator $A$,
the {\em zeta-function} of $A$ (whenever it is defined) is
$$
\zeta(s,A):=\Tr'A^{-s},
$$
where $\Tr'$ stands for the trace restricted to the subspace orthogonal to
$\ker(A)$.
If $\zeta(s,A)$ can be extended meromorphically in $s$ so that it is
holomorphic at $s=0$, then the {\em zeta-function regularized determinant}
of $A$ is defined as
$$
\Det' A=e^{-\zeta'(0,A)}.
$$
If $A$ is an elliptic differential operator of order $m$ on a compact manifold
of dimension $n$, then $\zeta(s,A)$ is holomorphic when $\Re(s)>n/m$ and can
be extended meromorphically to the entire complex plane with possible simple
poles at $\{\frac{n-l}{m},l=0,1,2,\dots\}$ only \cite{Se} (cf.~\cite{Sh}).
Moreover, the extended function is holomorphic at $s=0$ and
therefore the determinant $\Det'A$ is defined for such an operator.
Examples are the Laplacians $\Delta_i^\E$ acting on $i$-forms on a compact
Riemannian manifold $X$ with values in a vector bundle $\E$ with an Hermitian
structure; their determinants $\Det'\Delta_i^{\E}$ enter the Ray-Singer
analytic torsion for the de Rham complex \cite{RS,RS2}.
For the twisted de Rham complex (\ref{cplx}), the Laplacians 
$\Delta_{\bar k}=\Delta_{\bar k}^{\E,H}$ ($k=0,1$) acting on even/odd-degree
forms are also elliptic, and therefore the determinants $\Det'\Delta_{\bar k}$
($k=0,1$) still make sense (and are in fact equal).
However, what appear in the twisted analytic torsion (\ref{def-tor})
are not these determinants, but $\Det'd_{\bar k}^\dagger d_{\bar k}$,
which are much harder to define.

Let $\spec(A)$ ($\spec'(A)$, respectively) be the set of eigenvalues
(positive eigenvalues, respectively) of $A$.
For any $\lambda\in\spec(A)$, let $m(\lambda,A)$ be its multiplicity.
Then
$$
\zeta(s,A)=\sum_{\lambda\in\spec'(A)}\frac{m(\lambda,A)}{\lambda^s}.
$$
Given a flat vector bundle $\E$ over a manifold $X$ and a closed
odd-degree form $H$ on $X$, set $\spec\I(\Delta_{\bar 0})
:=\spec(\Delta_{\bar 0}|_{\im(d_{\bar 0}^\dagger)})
=\spec'(d_{\bar 0}^\dagger d_{\bar 0})$,
$m\I(\lambda,\Delta_{\bar 0}):=
m(\lambda,\Delta_{\bar 0}|_{\im(d_{\bar 0}^\dagger)})$ and
$\spec\II(\Delta_{\bar 0}):=\spec(\Delta_{\bar 0}|_{\im(d_{\bar 1})})
=\spec'(d_{\bar 1}^\dagger d_{\bar 1})$,
$m\II(\lambda,\Delta_{\bar0}):=m(\lambda,\Delta_{\bar0}|_{\im(d_{\bar1})})$.
Since $\Delta_{\bar0}$ is diagonal with respect to the decomposition
$C^{\bar0}=\im(d_{\bar 0}^\dagger)\oplus\im(d_{\bar1})\oplus
\ker(\Delta_{\bar0})$, we have 
$\spec'(\Delta_{\bar0})=\spec\I(\Delta_{\bar0})\cup\spec\II(\Delta_{\bar0})$
and $m(\lambda,\Delta_{\bar0})=m\I(\lambda,\Delta_{\bar0})
+m\II(\lambda,\Delta_{\bar0})$ if $\lambda>0$.
Therefore
\begin{equation}\label{part-zeta}
\zeta(s,d_{\bar0}^\dagger d_{\bar0})=
\!\!\sum_{\lambda\in\spec\I(\Delta_{\bar0})}\!\!
\frac{m\I(\lambda,\Delta_{\bar0})}{\lambda^s}, \qquad
\zeta(s,d_{\bar1}^\dagger d_{\bar1})=
\!\!\sum_{\lambda\in\spec\II(\Delta_{\bar0})}\!\!
\frac{m\II(\lambda,\Delta_{\bar0})}{\lambda^s}.
\end{equation}
The sum of the two zeta-functions is
\begin{equation}\label{sum-zeta}
\sum_{k=0,1}\zeta(s,d_{\bar k}^\dagger d_{\bar k})=\zeta(s,\Delta_{\bar0})
=\zeta(s,\Delta_{\bar1}).
\end{equation}
However, what we need for (\ref{def-tor}) is their difference.

\begin{theorem}\label{zeta-hol}
For $k=0,1$, $\zeta(s,d_{\bar k}^\dagger d_{\bar k})$ is holomorphic in
the half plane $\Re(s)>n/2$ and extends meromorphically to $\CO$ with
possible simple poles at $\{\frac{n-l}{2},l=0,1,2,\dots\}$ and possible
double poles at negative integers only, and is holomorphic at $s=0$.
\end{theorem}

\begin{proof}
Let $P_{\bar k}$ ($k=0,1$) be the orthogonal projection onto the closure
of the subspace $\im(d_{\bar k}^\dagger)$.
As $d_{\bar k}d_{\bar k}^\dagger$ and $\Delta_{\overline{k+1}}$ are
equal and invertible on (the closure of) $\im(d_{\bar k})$, we have 
$$
P_{\bar k}=d_{\bar k}^\dagger(d_{\bar k}d_{\bar k}^\dagger)^{-1}d_{\bar k}
=d_{\bar k}^\dagger(\Delta_{\overline{k+1}})^{-1}d_{\bar k},
$$
which implies that $P_{\bar k}$ is a pseudodifferential operator of order $0$.
Moreover,
$$
\zeta(s,d_{\bar k}^\dagger d_{\bar k})=\Tr(P_{\bar k}\Delta_{\bar k}^{-s}).
$$
By general theory \cite{GrSe,Gr06}, $\zeta(s,d_{\bar k}^\dagger d_{\bar k})$
is holomorphic in the half plane $\Re(s)>n/2$ and extends meromorphically
to $\CO$ with possible simple poles at $\{\frac{n-l}{2},l=0,1,2,\dots\}$ and
possible double poles at negative integers only.
The Laurent series of $\zeta(s,d_{\bar k}^\dagger d_{\bar k})$ at $s=0$ is
$$
\Tr(P_{\bar k}\Delta_{\bar k}^{-s})=
\frac{c_{-1}(P_{\bar k},\Delta_{\bar k})}{s}+c_0(P_{\bar k},\Delta_{\bar k})
+\sum_{l=1}^\infty c_l(P_{\bar k},\Delta_{\bar k})\,s^l.
$$
Here $c_{-1}(P_{\bar k},\Delta_{\bar k})=\frac{1}{2}\res(P_{\bar k})$,
where $\res(P_{\bar k})$ is known as the non-commutative residue or
the Guillemin-Wodzicki residue trace of $P_{\bar k}$ \cite{Wo,Gui}.
Since $P_{\bar k}$ is a projection, $\res(P_{\bar k})=0$
\cite{Wo,BrLe,Gr03}.
Therefore $\zeta(s,d_{\bar k}^\dagger d_{\bar k})$ is regular at $s=0$.
\end{proof}

Theorem~\ref{zeta-hol} justifies the definition of the twisted analytic
torsion in (\ref{def-tor}).
The constant term 
$\zeta(0,d_{\bar k}^\dagger d_{\bar k})=c_0(P_{\bar k},\Delta_{\bar k})$
of the above Laurent series is related to the Kontsevich-Vishik trace
\cite{KV,Gr06}.
It can nevertheless be studied by standard heat kernel techniques
(Lemma~\ref{sum-zeta-b}, Corollaries~\ref{var-zeta0} and \ref{var-zeta0-H}
below).
Recall the notion of twisted Betti numbers $b_{\bar k}(X,\E,H)$ ($k=0,1$)
from \S\ref{sect:twistedDR} and the fact that the zeta-function is related
to the heat kernel by a Mellin transform
$$
\zeta(s,A)=\frac{1}{\Gamma(s)}\int_0^\infty t^{s-1}\Tr'e^{-tA}\,dt.
$$

\begin{lemma}\label{sum-zeta-b}
If $\dim X$ is odd, then
$$
\sum_{k=0,1}\zeta(0,d_{\bar k}^\dagger d_{\bar k})=-b_{\bar0}(X,\E,H)=
-b_{\bar1}(X,\E,H).
$$
\end{lemma}

\begin{proof}
When $n=\dim X$ is odd, $b_{\bar0}(X,\E,H)=b_{\bar1}(X,\E,H)$ as
$\chi(X,\E,H)=0$.
By (\ref{sum-zeta}), it suffices to show that 
$\zeta(0,\Delta_{\bar k})=-b_{\bar k}(X,\E,H)$, $k=0,1$.
By the asymptotic expansion of the heat kernel (cf.~\cite{Gil,BGV}),
$$
\Tr e^{-t\Delta_{\bar k}}\sim\sum_{l=0}^\infty c_{\bar k,l}t^{-n/2+l}
$$
as $t\downarrow 0$, where $c_{\bar k,l}\in\RE$.
We have, for $\Re(s)>n/2$,
\begin{align}
\zeta(s,\Delta_{\bar k})&=\frac{1}{\Gamma(s)}
 \int_0^\infty t^{s-1}(\Tr e^{-t\Delta_{\bar k}}-b_{\bar k})\,dt      \nno
&=\frac{1}{\Gamma(s)}\Big(-\frac{b_{\bar k}}{s}+\sum_{l=0}^N
 \frac{c_{\bar k,l}}{s-n/2+l}\,+R_N(s)\Big),                     \nonumber
\end{align}
where $N$ is a sufficiently large integer.
Here $R_N(s)$ is holomorphic when $\Re(s)>n/2-N$. 
Since $n$ is odd and since the Gamma function $\Gamma(s)$ has a simple pole
at $s=0$, the result follows.
\end{proof}

\section{Twisted analytic torsion under metric and flux deformations}
\label{sect:var}

\subsection{Variation of analytic torsion with respect to the metrics}
We assume that $X$ is a compact oriented manifold of odd dimension.
Let $g_X$ be a Riemannian metric on $X$ and $g_\E$, an Hermitian metric 
on $\E$.
Let $Q_{\bar k}$ ($k=0,1$) be the orthogonal projection from (the completion
of) $C^{\bar k}$ to $\ker(\Delta_{\bar k})$. 
Suppose that the pair $(g_X,g_\E)$ is deformed smoothly along a one-parameter
family with parameter $u\in\RE$, then the operators $\ast$, $\sharp$ and 
$\Gamma=\ast\sharp=\sharp\ast$ (see \$\ref{sect:scalar}) all depend smoothly
on $u$.
Let 
$$
\alpha=\Gamma^{-1}\frac{\partial\Gamma}{\partial u}.
$$
We show the invariance of the analytic torsion (\ref{def-tor}) by showing 
in the next two lemmas that the variation of the regularized determinants
cancels that of the volume elements.

\begin{lemma}\label{lem:var-det}
Under the above assumptions,
$$
\frac{\partial}{\partial u}\log[\Det'd_{\bar 0}^\dagger d_{\bar 0}\,
(\Det'd_{\bar 1}^\dagger d_{\bar 1})^{-1}]
=\sum_{k=0,1}(-1)^k\Tr(\alpha Q_{\bar k}).
$$
\end{lemma}

\begin{proof}
While $d_{\bar k}$ is independent of $u$, we have
$$
\frac{\partial d_{\bar k}^\dagger}{\partial u}=-[\alpha,d_{\bar k}^\dagger],
$$
which follows easily from Lemma~\ref{lemma:adj}.
Using $P_{\bar k}d_{\bar k}^\dagger=d_{\bar k}^\dagger$, 
$d_{\bar k}P_{\bar k}=d_{\bar k}$ and $P_{\bar k}^2=P_{\bar k}$,
we get $d_{\bar k}^\dagger d_{\bar k}P_{\bar k}=P_{\bar k}
d_{\bar k}^\dagger d_{\bar k}=d_{\bar k}^\dagger d_{\bar k}$ and
$$
\frac{\partial P_{\bar k}}{\partial u}=
\frac{\partial P_{\bar k}}{\partial u}P_{\bar k},
\qquad P_{\bar k}\frac{\partial P_{\bar k}}{\partial u}=0.
$$
Following the $\ZA$-graded case \cite{RS,RS2}, we set
\begin{align}
f(s,u)&=\sum_{k=0,1}(-1)^k\int_0^\infty t^{s-1}
 \Tr(e^{-td_{\bar k}^\dagger d_{\bar k}}P_{\bar k})\,dt         \nno
&=\Gamma(s)\sum_{k=0,1}(-1)^k\zeta(s,d_{\bar k}^\dagger d_{\bar k}). \nonumber
\end{align}
Using the above identities on $P_{\bar k}$, the trace property and by an
application of Duhamel's principle, we get
\begin{align}
\frac{\partial f}{\partial u} 
&=\sum_{k=0,1}(-1)^k\int_0^\infty\!t^{s-1}\Tr\left(
  t[\alpha,d_{\bar k}^\dagger]d_{\bar k}e^{-td_{\bar k}^\dagger d_{\bar k}}
  +e^{-t d_{\bar k}^\dagger d_{\bar k}}
  \frac{\partial P_{\bar k}}{\partial u}P_{\bar k}\right)\,dt \nno
&=\sum_{k=0,1}(-1)^k\int_0^\infty\!t^{s-1}\Tr\left(t\alpha[d_{\bar
      k}^\dagger,
  d_{\bar k}e^{-td_{\bar k}^\dagger d_{\bar k}}]
  +P_{\bar k}e^{-td_{\bar k}^\dagger d_{\bar k}}
  \frac{\partial P_{\bar k}}{\partial u}\right)dt \nno
&=\sum_{k=0,1}(-1)^k\int_0^\infty\!t^{s-1}\Tr\left(
  t\alpha(e^{-td_{\bar k}^\dagger d_{\bar k}}d_{\bar k}^\dagger d_{\bar k}
  -e^{-td_{\bar k}d_{\bar k}^\dagger}d_{\bar k}d_{\bar k}^\dagger)
  +e^{-td_{\bar k}^\dagger d_{\bar k}}P_{\bar k}
  \frac{\partial P_{\bar k}}{\partial u}\right)dt \nno
&=\sum_{k=0,1}(-1)^k\int_0^\infty\!t^s
  \Tr(\alpha e^{-t\Delta_{\bar k}}\Delta_{\bar k})\,dt \nno
&=-\sum_{k=0,1}(-1)^k\int_0^\infty\!t^s\frac{\partial}{\partial t}
  \Tr\big(\alpha(e^{-t\Delta_{\bar k}}-Q_{\bar k})\big)\,dt.   \nonumber
\end{align}
Integrating by parts, we have
\begin{align}
\frac{\partial f}{\partial u}
&=s\sum_{k=0,1}(-1)^k\int_0^\infty\!t^{s-1}
 \Tr\big(\alpha(e^{-t\Delta_{\bar k}}-Q_{\bar k})\big)\,dt \nno
&=s\sum_{k=0,1}(-1)^k\left(\int_0^1+\int_1^\infty\right)t^{s-1}
 \Tr\big(\alpha(e^{-t\Delta_{\bar k}}-Q_{\bar k})\big)\,dt. \nonumber
\end{align}
Since $\alpha$ is a smooth tensor and $n$ is odd, the asymptotic expansion
as $t\downarrow 0$ for $\Tr(\alpha e^{-t\Delta_{\bar k}})$ does not contain
a constant term (see for example \cite{Gil}, Lemma~1.7.4).
Therefore $\int_0^1 t^{s-1}\Tr(\alpha e^{-t\Delta_{\bar k}})\,dt$
does not have a pole at $s=0$.
On the other hand, because of the exponential decay of 
$\Tr(\alpha(e^{-t\Delta_{\bar k}}-Q_{\bar k}))$ for large $t$, the function
$\int_1^\infty t^{s-1}\Tr(\alpha(e^{-t\Delta_{\bar k}}-Q_{\bar k}))\,dt$
is entire in $s$.
So
\begin{equation}\label{var-f}
\frac{\partial f}{\partial u}\Big|_{s=0}=-s\sum_{k=0,1}(-1)^k
  \int_0^1 t^{s-1}\Tr(\alpha Q_{\bar k})\,dt\Big|_{s=0}         
=-\sum_{k=0,1}(-1)^k\Tr(\alpha Q_{\bar k})            
\end{equation}
and hence
\begin{equation}\label{var-diffzeta0}
\frac{\partial}{\partial u}\sum_{k=0,1}(-1)^k
\zeta(0,d_{\bar k}^\dagger d_{\bar k})=0.
\end{equation}
Finally, the result follows from (\ref{var-f}), (\ref{var-diffzeta0}) and
$$
\log[\Det'd_{\bar 0}^\dagger d_{\bar 0}\,(\Det'd_{\bar 1}^\dagger
d_{\bar 1})^{-1}]=-\lim_{s\to0}\Big[f(s,u)-\frac{1}{s}\sum_{k=0,1}(-1)^k
\zeta(0,d_{\bar k}^\dagger d_{\bar k})\Big].
$$

\vspace{-10pt}
\end{proof}

\begin{corollary}\label{var-zeta0}
Under the above deformation, each $\zeta(0,d_{\bar k}^\dagger d_{\bar k})$
($k=0,1$) is invariant.
\end{corollary}

\begin{proof}
By (\ref{var-diffzeta0}), their difference is invariant.
By Lemma~\ref{sum-zeta-b}, their sum is also invariant since 
$b_{\bar0}(X,\E,H)$ is defined without using the metrics.
\end{proof}

\begin{lemma}\label{lem:var-vol}
Under the same assumptions, along any one-parameter deformation of
$(g_X,g_\E)$, the volume elements $\eta_{\bar0}$, $\eta_{\bar1}$
can be chosen so that 
$$
\frac{\partial}{\partial u}(\eta_{\bar0}\otimes\eta_{\bar1}^{-1})=-\frac{1}{2}
\sum_{k=0,1}(-1)^k\Tr(\alpha Q_{\bar k})\,\eta_{\bar0}\otimes\eta_{\bar1}^{-1}.
$$
\end{lemma}

\begin{proof}
Recall that 
$\eta_{\bar k}=\nu_{\bar k,1}\wedge\cdots\wedge\nu_{\bar k,b_{\bar k}}$,
where $\{\nu_{\bar k,i}\}_{i=1}^{b_{\bar k}}$ is an orthonormal basis of
$H^{\bar k}(X,\E,H)$, $k=0,1$.
Since $(\nu_{\bar k,i},\nu_{\bar k,j})
=\int_X\nu_{\bar k,i}\wedge\Gamma\,\nu_{\bar k,j}=\delta_{ij}$,
we get, by taking the derivative with respect to $u$,
$$
\Re\left(\frac{\partial\nu_{\bar k,i}}{\partial u},\nu_{\bar k,i}\right)
=-\frac{1}{2}(\nu_{\bar k,i},\alpha\nu_{\bar k,i}).
$$
We can adjust the phase of $\nu_{\bar k,i}$ so that 
$(\frac{\partial\nu_{\bar k,i}}{\partial u},\nu_{\bar k,i})$ is real.
Since we identify $\det\ker(\Delta_{\bar k})$ with $\det H^{\bar k}(X,\E,H)$
along the deformation, we have
\begin{align}
\frac{\partial\eta_{\bar k}}{\partial u}
&=\sum_{i=1}^{b_{\bar k}}\nu_{\bar k,1}\wedge\cdots\wedge
  \frac{\partial\nu_{\bar k,i}}{\partial u}\wedge\cdots
  \wedge\nu_{\bar k,b_{\bar k}} \nno
&=-\frac{1}{2}\sum_{i=1}^{b_{\bar k}}(\nu_{\bar k,i},\alpha\,\nu_{\bar k,i})
  \,\eta_{\bar k} 
=-\frac{1}{2}\Tr(\alpha Q_{\bar k})\,\eta_{\bar k}. \nonumber
\end{align}
The result follows.
\end{proof}

Combining Lemma~\ref{lem:var-det} and Lemma~\ref{lem:var-vol}, we have

\begin{theorem}[metric independence of analytic torsion]\label{thm:indept}
Let $X$ be a compact, oriented manifold of odd dimension, $\E$, a flat vector
bundle over $X$, and $H$, a closed differential form on $X$ of odd degree.
Then the analytic torsion $\tau(X,\E,H)$ of the twisted de Rham complex
does not depend on the choice of the Riemannian metric on $X$ or the
Hermitian metric on $\E$.
\end{theorem} 

\subsection{Variation of analytic torsion with respect to the flux in a
cohomology class}\label{sect:var-H}
We continue to assume that $\dim X$ is odd and use the same notation as above.
Suppose the (real) flux form $H$ is deformed smoothly along a one-parameter
family with parameter $v\in\RE$ in such a way that the cohomology class
$[H]\in H^{\bar1}(X,\RE)$ is fixed.
Then $\frac{\partial H}{\partial v}=-dB$ for some form $B\in\Omega^{\bar0}(X)$
that depends smoothly on $v$; let
$$
\beta=B\wedge\cdot\;.
$$
As before, we show in the next two lemmas that the variation of the 
regularized determinants cancels that of the volume elements.

\begin{lemma}\label{lem:var-det-H}
Under the above assumptions,
$$
\frac{\partial}{\partial v}\log[\Det'd_{\bar 0}^\dagger d_{\bar 0}\,
(\Det'd_{\bar 1}^\dagger d_{\bar 1})^{-1}]
=2\sum_{k=0,1}(-1)^k\Tr(\beta Q_{\bar k}).
$$
\end{lemma}

\begin{proof}
As in the proof of Lemma~\ref{lem:var-det}, we set
$$
f(s,v)=\sum_{k=0,1}(-1)^k\int_0^\infty t^{s-1}
\Tr(e^{-td_{\bar k}^\dagger d_{\bar k}}P_{\bar k})\,dt.
$$
We note that $B$, hence $\beta$ is real.
Using
$$
\frac{\partial d_{\bar k}}{\partial v}=[\beta,d_{\bar k}],
\qquad\frac{\partial d_{\bar k}^\dagger}{\partial v}=-
[\beta^\dagger,d_{\bar k}^\dagger],
$$
$$
P_{\bar k}^2=P_{\bar k}=P_{\bar k}^\dagger,\qquad
P_{\bar k}\frac{\partial P_{\bar k}}{\partial v}P_{\bar k}=0
$$
and by Dumahel's principle, we get
\begin{align}
\frac{\partial f}{\partial v} 
&=\sum_{k=0,1}(-1)^k\int_0^\infty\!t^{s-1}
  \Tr\left(t([\beta^\dagger,d_{\bar k}^\dagger]d_{\bar k}-
  d_{\bar k}^\dagger[\beta,d_{\bar k}])
  e^{-td_{\bar k}^\dagger d_{\bar k}}+e^{-t d_{\bar k}^\dagger d_{\bar k}}
  \{{\textstyle \frac{\partial P_{\bar k}}{\partial v}},P_{\bar k}\}
  \right)dt \nno
&=2\sum_{k=0,1}(-1)^k\int_0^\infty\!t^{s-1}\Tr\left(
  t\beta(e^{-td_{\bar k}^\dagger d_{\bar k}}d_{\bar k}^\dagger d_{\bar k}
  -e^{-td_{\bar k}d_{\bar k}^\dagger}d_{\bar k}d_{\bar k}^\dagger)
  +e^{-td_{\bar k}^\dagger d_{\bar k}}P_{\bar k}
  {\textstyle \frac{\partial P_{\bar k}}{\partial v}}P_{\bar k}\right)dt \nno
&=2\sum_{k=0,1}(-1)^k\int_0^\infty\!t^s
  \Tr(\beta e^{-t\Delta_{\bar k}}\Delta_{\bar k})\,dt \nno
&=-2\sum_{k=0,1}(-1)^k\int_0^\infty\!t^s\frac{\partial}{\partial t}
  \Tr\big(\beta(e^{-t\Delta_{\bar k}}-Q_{\bar k})\big)\,dt.   \nonumber
\end{align}
The rest is similar to the proof of Lemma~\ref{lem:var-det}. 
\end{proof}

\begin{corollary}\label{var-zeta0-H}
Under the above deformation, each $\zeta(0,d_{\bar k}^\dagger d_{\bar k})$
($k=0,1$) is invariant.
\end{corollary}

\begin{proof}
We follow the proof of Corollary~\ref{var-zeta0}, using the fact that
$b_{\bar0}(X,\E,H)$ depends only on the cohomology class of $H$.
\end{proof}

If $n$ is odd and $H=0$, then $\zeta(0,\Delta_i^{\E})=-b_i(X,\E)$ and
\begin{align}
\zeta(0,d_{\bar0}^\dagger d_{\bar0})&=\sum_{i=1}^{(n+1)/2}
        i\,(b_{2i}(X,\E)-b_{2i-1}(X,\E)),                      \nno
\zeta(0,d_{\bar1}^\dagger d_{\bar1})&=\sum_{i=1}^{(n-1)/2}
        i\,(b_{2i+1}(X,\E)-b_{2i}(X,\E)).                      \nonumber
\end{align}
We hope that when $n$ is odd but $H\ne0$, 
$\zeta(0,d_{\bar k}^\dagger d_{\bar k})$ can also be expressed in terms
of topological numbers that are invariant under homotopy equivalences 
preserving $[H]$.

\begin{lemma}\label{lem:var-vol-H}
Under the same assumptions, along any one-parameter deformation of $H$
that fixes the cohomology class $[H]$, the volume elements $\eta_{\bar0}$,
$\eta_{\bar1}$ can be chosen so that 
$$
\frac{\partial}{\partial v}(\eta_{\bar0}\otimes\eta_{\bar1}^{-1})=
-\sum_{k=0,1}(-1)^k\Tr(\beta Q_{\bar k})\,\eta_{\bar0}\otimes\eta_{\bar1}^{-1},
$$
where we identify $\det H^\bullet(X,\E,H)$ along the deformation using
(\ref{e^B}).
\end{lemma}

\begin{proof}
Fix a reference point, say $v=0$, and let $H^{(0)}$, $\eta_{\bar k}^{(0)}$
be the values of $H$, $\eta_{\bar k}$, respectively, at $v=0$.
To compare the volume elements $\eta_{\bar k}\in\det H^{\bar k}(X,\E,H)$ 
at different values of $v$, we map them to $\det H^{\bar k}(X,\E,H^{(0)})$
by the inverse of the isomorphism 
$$
\det\varepsilon_B\colon\det H^\bullet(X,\E,H^{(0)})\to\det H^\bullet(X,\E,H)
$$ 
induced by (\ref{e^B}).
Since $\varepsilon_B=e^\beta$ on $\Omega^\bullet(X,\E)$, we have, for $k=0,1$, 
$$
\frac{\partial}{\partial v}(\det\varepsilon_B)^{-1}\eta_{\bar k}=
-\Tr(\beta Q_{\bar k})\,(\det\varepsilon_B)^{-1}\eta_{\bar k}.
$$
The result follows.
\end{proof}

Combining Lemma~\ref{lem:var-det-H} and Lemma~\ref{lem:var-vol-H}, we have

\begin{theorem}[flux representative independence of analytic torsion]
\label{thm:indept-H}
Let $X$ be a compact, oriented manifold of odd dimension and $\E$,
a flat vector bundle over $X$.
Suppose $H$ and $H'$ are closed differential forms on $X$ of odd degrees
representing the same de Rham cohomology class, and let $B$ be an even
form so that $H'=H-dB$.
Then the analytic torsion $\tau(X,\E,H')=(\det\varepsilon_B)(\tau(X,\E,H))$.
\end{theorem} 

\subsection{Relation to generalized geometry}
Recall that in generalized geometry \cite{Hi,Gua1}, the bundle
$TX\oplus T^*X$ has a bilinear form of signature $(n,n)$ given by
$$
\langle\xi_1+W_1,\xi_2+W_2\rangle:=(\xi_1(W_2)+\xi_2(W_1))/2,
$$
where for $i=1,2$, $\xi_i$ are $1$-forms and $W_i$ are vector fields on $X$. 
A {\em generalized metric} on $X$ is a reduction of the structure group
$O(n,n)$ to $O(n)\times O(n)$.
Equivalently, a generalized metric is a splitting of $TX\oplus T^*X$ to
a direct sum of two sub-bundles of rank $n$ so that the bilinear form is
positive on one and negative on the other.
The positive sub-bundle is the graph of $g+B\in\Gamma(\Hom(TX,T^*X))$, where
$g=g_X$ is a usual Riemannian metric on $X$ and $B$ is a $2$-form on $X$. 

A generalized metric on $X$ defines as follows a scalar product, called
the Born-Infeld metric \cite{Gua2}, on $\Omega^\bullet(X)$.
Let $\sigma$ be the isomorphism from $\Omega^\bullet(X)$ to itself so that
if $\omega$ is the wedge product of $1$-forms, then $\sigma(\omega)$ is 
the product with the order of $1$-forms reversed.
Thus $\sigma(\omega\wedge\omega')=\sigma(\omega')\wedge\sigma(\omega)$ for
any forms $\omega$, $\omega'$ and $\sigma(B)=-B$ if $B$ is a $2$-form.
Choose a (local) orthonormal frame $\{e_i,i=1,\dots,n\}$ on $X$ with respect
to $g$ and let $\hat e_i:=\iota_{e_i}+\iota_{e_i}(g+B)\wedge\cdot$
($i=1,\dots,n$) be operators acting on forms.
Define a new star operation by
$$
\star_{(g,B)}\,\omega=\sigma(\hat e_n\cdots\hat e_2\hat e_1\omega).
$$
When $B=0$, $\star_{(g,B)}$ is the usual Hodge star $*_g$ given by $g$.
The {\em Born-Infeld metric} (scalar product) on $\Omega^\bullet(X)^\CO$ is
\cite{Gua2} 
$$
(\omega,\omega')_{(g,B)}:=\int_X\omega\wedge\star_{(g,B)}\overline{\omega'}
$$
for $\omega,\omega'\in\Omega^\bullet(X)^\CO$.

We show that the isomorphism $\varepsilon_B$ intertwines the 
Born-Infeld metric $(\cdot,\cdot)_{(g,B)}$ and the usual scalar
product $(\cdot,\cdot)_g$ defined by the Riemannian metric $g$.

\begin{lemma}\label{prop:ggeom}
For any $\omega,\omega'\in\Omega^\bullet(X)^\CO$, we have
$$
(\omega,\omega')_{(g,B)}=(\varepsilon_B(\omega),\varepsilon_B(\omega'))_g.
$$
\end{lemma}

\begin{proof}
Since
$$
(\varepsilon_B^{-1}(\omega),\varepsilon_B^{-1}(\omega'))_{(g,B)}=
\int_X\omega\wedge\sigma(\varepsilon_B\,\hat e_n\cdots\hat e_1\,
\varepsilon_B^{-1}\overline{\omega'}),
$$
it suffices to check that 
$\varepsilon_B\,\hat e_n\cdots\hat e_1\,\varepsilon_B^{-1}$
is independent of $B$.
We replace $B$ by $vB$, where $v\in\RE$.
Then, since $\frac{\partial\hat e_i}{\partial v}=\iota_{e_i}B\wedge\cdot
=-[\beta,\hat e_i]$, we get
$$
\frac{\partial}{\partial v}
\big(\varepsilon_{vB}\,\hat e_n\cdots\hat e_1\,\varepsilon_{vB}^{-1}\big)
=\varepsilon_{vB}\Big([\beta,\hat e_n\cdots\hat e_1]+\sum_{i=1}^n
 \hat e_n\cdots\frac{\partial\hat e_i}{\partial v}\cdots\hat e_1\Big)
 \varepsilon_{vB}^{-1}=0
$$
and the result follows.
\end{proof}

For simplicity, we take $\E$ as the trivial line bundle over $X$.
Let $(d^H_{\bar k})_{(g,B)}^\dagger$ be the adjoint of
$d_{\bar k}^H=d+H\wedge\,\cdot\,$ acting on $\Omega^{\bar k}(X)$ ($k=0,1$)
with respect to the Born-Infeld metric \cite{Gua2}.
Let $(\Delta^H_{\bar k})_{(g,B)}=(d^H_{\bar k})_{(g,B)}^\dagger d_{\bar k}^H
+d_{\overline{k+1}}^H(d^H_{\overline{k+1}})_{(g,B)}^\dagger$ be the
corresponding Laplacians.
Generalizing (\ref{def-tor}) in \S\ref{subsect:constr}, we define the
{\em twisted analytic torsion under a generalized metric} $(g,B)$ as
\[ \tau_{(g,B)}(X,H):=
\big(\Det'(d^H_{\bar 0})_{(g,B)}^\dagger d^H_{\bar 0}\big)^{1/2}
\big(\Det'(d^H_{\bar 1})_{(g,B)}^\dagger d_{\bar 1}\big)^{-1/2}
(\eta^H_{\bar 0})_{(g,B)}\otimes(\eta^H_{\bar 1})_{(g,B)}^{-1}
\in\det H^\bullet(X,H),                                                   \]
where, for $k=0,1$, the determinant of
$(d^H_{\bar k})_{(g,B)}^\dagger d^H_{\bar k}$ is defined as before due to the
vanishing of the non-commutative residue of the pseudodifferential projection
onto $\im(d^H_{\bar k})_{(g,B)}^\dagger$ and $(\eta^H_{\bar k})_{(g,B)}$ is the
unit volume element of $\ker(\Delta^H_{\bar k})_{(g,B)}\cong H^{\bar k}(X,H)$
with respect to the Born-Infeld metric.
When $B=0$, the torsion is $\tau_g(X,H)$ defined in (\ref{def-tor}).

We specialize to the interesting case when $H$ is a $3$-form.

\begin{proposition}\label{prop:gconj}
Let $H$ be a closed $3$-form on $X$ and $H'=H-dB$, where $B$ is the $2$-form
in the generalized metric.
Then 
\[ (d^{H'})_g^\dagger=\varepsilon_B\circ(d^H)_{(g,B)}^\dagger\circ
\varepsilon_B^{-1},\quad
\Delta^{H'}_g=\varepsilon_B\circ\Delta^H_{(g,B)}\circ\varepsilon_B^{-1}.  \]
\end{proposition}

\begin{proof}
We have $d^{H'}=\varepsilon_B\circ d^H\circ\varepsilon_B^{-1}$ from
\$\ref{sect:twistedDR}.
By Lemma~\ref{prop:ggeom}, for any $\omega,\omega'\in\Omega^\bullet(X)^\CO$,
we get
\begin{align*}
\big(\omega,(d^H)_{(g,B)}^\dagger\omega'\big)_{(g,B)}
&=\big(d^H\omega,\omega'\big)_{(g,B)}
=\big(\varepsilon_Bd^H\omega,\varepsilon_B\omega'\big)_g 
=\big(d^{H'}\varepsilon_B\omega,\varepsilon_B\omega'\big)_g\\
&=\big(\varepsilon_B\omega,(d^{H'})^\dagger_g\varepsilon_B\omega'\big)_g
=\big(\omega,\varepsilon_B^{-1}(d^{H'})^\dagger_g
\varepsilon_B\omega'\big)_{(g,B)}.
\end{align*}
Therefore the first equality holds and the second one follows.
\end{proof}

\begin{corollary}
Under the above notations, we have
\[  \tau_g(X,H')=(\det\varepsilon_B)
\big(\tau_{(g,B)}(X,H)\big)\in\det H^\bullet(X,H').   \]
\end{corollary}

\begin{proof}
By Proposition~\ref{prop:gconj}, for $k=0,1$, the operators 
$(d^{H'}_{\bar k})_g^\dagger d^{H'}_{\bar k}$ and
$(d^H_{\bar k})_{(g,B)}^\dagger d^H_{\bar k}$ have the same spectrum and
hence the same regularized determinants.
On the other hand, together with Lemma~\ref{prop:ggeom}, we conclude that
$\varepsilon_B\colon\ker(\Delta^H_{\bar k})_{(g,B)}\to
\ker(\Delta^{H'}_{\bar k})_g$ is an isometry and hence we can choose
$(\eta^{H'}_{\bar k})_g=(\det\varepsilon_B)(\eta^H_{\bar k})_{(g,B)}$.
The result follows. 
\end{proof}

We thus conclude that deformation of $H$ by a $B$-field is equivalent
to deforming the usual metric to a generalized metric.
In this way, deformations of the usual metric and that of the flux by a
$B$-field are unified to a deformation of generalized metric. 
Theorems~\ref{thm:indept} and \ref{thm:indept-H} state that the torsion is
invariant under such a deformation.

\section{Functorial properties of analytic torsion}\label{sect:func}

In this section, we state the basic functorial properties of analytic
torsion for the twisted de Rham complex.
These can be established by a generalization of the proofs of the
corresponding results for the usual analytic torsion \cite{RS,C79,M93}
to the $\ZA_2$-graded case.
We write $d_{\bar k}^\E=d_{\bar k}^{\E,H}$,
$\Delta_{\bar k}^\E=\Delta_{\bar k}^{\E,H}$ and 
$\eta_{\bar k}^\E=\eta_{\bar k}^{\E,H}$ since 
the dependence on the flux form $H$ is clear.

\begin{proposition}
Let $X$ be a compact, oriented Riemannian manifold and $\E_1,\E_2$,
flat Hermitian vector bundles on $X$.
Suppose $H$ is a closed odd-degree form on $X$.
Then 
$$
\tau(X,\E_1\oplus\E_2,H)=\tau(X,\E_1,H)\otimes\tau(X,\E_2,H) 
$$
under the canonical identification
$$
\det H^\bullet(X,\E_1\oplus\E_2,H)\cong\det H^\bullet(X,\E_1,H)\otimes
\det H^\bullet(X,\E_2,H)
$$
induced by the isomorphism
$H^\bullet(X,\E_1\oplus\E_2,H)\cong
H^\bullet(X,\E_1,H)\oplus H^\bullet(X,\E_2,H)$.
\end{proposition}

\begin{proof}
On $\Omega^\bullet(X,\E_1\oplus\E_2)\cong\Omega^\bullet(X,\E_1)\oplus
\Omega^\bullet(X,\E_2)$, the operator $d_{\bar k}^{\E_1\oplus\E_2}=
d_{\bar k}^{\E_1}\oplus d_{\bar k}^{\E_2}$ is block-diagonal.
Thus the determinant factorizes:
$\Det'((d_{\bar k}^{\E_1\oplus\E_2})^\dagger d_{\bar k}^{\E_1\oplus\E_2})
=\Det'((d_{\bar k}^{\E_1})^\dagger d_{\bar k}^{\E_1})
\Det'((d_{\bar k}^{\E_2})^\dagger d_{\bar k}^{\E_2})$.
Under the above identification, we can choose the volume elements such that
$\eta_{\bar k}^{\E_1\oplus\E_2}=
\eta_{\bar k}^{\E_1}\otimes\eta_{\bar k}^{\E_2}$.
Hence the result.
\end{proof}

\begin{proposition}\label{prop:dual}
Let $X$ be a compact, oriented manifold of dimension $n$ and $\E$, a flat
vector bundle on $X$. 
Suppose $H$ is a closed odd-degree form on $X$.
Then 
$$
\tau(X,\E,H)=\tau(X,\E^*\!,-H)^{(-1)^{n+1}}
$$
under the canonical identification
$$
\det H^\bullet(X,\E,H)\cong\det H^\bullet(X,\E^*\!,-H)^{(-1)^{n+1}} 
$$
induced by Poincar\'e duality in Proposition~\ref{prop:pd}. 
\end{proposition}

\begin{proof}
By Lemma~\ref{lemma:adj}, $(d_{\bar k}^{\E,H})^\dagger d_{\bar k}^{\E,H}
=\Gamma^{-1}d_{\overline{n+1-k}}^{\E^*\!,-H}
(d_{\overline{n+1-k}}^{\E^*\!,-H})^\dagger\Gamma$.
So the non-zero spectrum of
$(d_{\bar k}^{\E,H})^\dagger d_{\bar k}^{\E,H}$,
counting multiplicity, is identical to that of 
$(d_{\overline{n+1-k}}^{\E^*\!,-H})^\dagger d_{\overline{n+1-k}}^{\E^*\!,-H}$,
and so is the regularized determinant.
The isometry $\Gamma$ induces Poincar\'s duality, under which the volume
elements $\eta_{\bar k}^{\E,H}=\eta_{\overline{n+1-k}}^{\E^*\!,-H}$.
The result then follows from the definition of the twisted torsion.
\end{proof}

\begin{proposition}
Let $X_1$, $X_2$ be two compact oriented manifolds with the same universal
covering manifold.
Suppose the fundamental group $\pi_1(X_1)$ is a subgroup of $\pi_1(X_2)$.
Let $\rho_1$ be a representation of $\pi_1(X_1)$ and let $\rho_2$ be the
induced representation of $\pi_1(X_2)$.
Denote by the flat vector bundles associated with $\rho_1$, $\rho_2$ by
$\E_1$, $\E_2$, respectively.
Suppose the closed odd-degree forms $H_1$ on $X_1$ and $H_2$ on $X_2$
pull-back to the same form on the universal covering.
Then 
$$
\tau(X_1,\E_1,H_1)=\tau(X_2,\E_2,H_2)
$$
under the canonical identification 
$$
\det H^\bullet(X_1,\E_1,H_1)\cong\det H^\bullet(X_2,\E_2,H_2)
$$
induced by the isomorphism
$H^\bullet(X_1,\E_1,H_1)\cong H^\bullet(X_2,\E_2,H_2)$.
\end{proposition}

\begin{proof}
By Theorem~\ref{thm:indept}, we can choose the Riemannian metrics
on $X_1$ and $X_2$ so that they pull-back to the same metric on the 
universal covering and the Hermitian metrics on $\E_1$ and $\E_2$ 
associated to metric on the space of representation $\rho_1$ and 
the induced metric on the space of representation $\rho_2$.
Then the canonical isomorphism 
$\Omega^\bullet(X_1,\E_1,H_1)\cong\Omega^\bullet(X_2,\E_2,H_2)$
is an isometry.
Following the proof Theorem~2.6 in \cite{RS}, we deduce that the
spectrums, and hence the regularized determinants of 
$(d_{\bar k}^{\E_1})^\dagger d_{\bar k}^{\E_1}$ and 
$(d_{\bar k}^{\E_2})^\dagger d_{\bar k}^{\E_2}$ coincide.
The volume elements of $H^\bullet(X_1,\E_1,H_1)$ and 
$H^\bullet(X_2,\E_2,H_2)$ also coincide under the isomorphism.
\end{proof}

\begin{proposition}
In addition to the conditions of Proposition~\ref{prop:ku}, assume that both
$X_1$ and $X_2$ are compact manifolds.
Then under the canonical identification
\begin{align}
\det H^\bullet(X_1\times X_2,&\E_1\boxtimes\E_2,H_1\boxplus H_2) \nno
&\cong(\det H^\bullet(X_1,\E_1,H_1))^{\otimes\chi(X_2,\E_2)}\otimes
  (\det H^\bullet(X_2,\E_2,H_2))^{\otimes\chi(X_1,\E_1)} \nonumber
\end{align}
induced by the K\"unneth isomorphism, we have
$$
\tau(X_1\times X_2,\E_1\boxtimes\E_2,H_1\boxplus H_2)=
\tau(X_1,\E_1,H_1)^{\otimes\chi(X_2,\E_2)}\otimes
\tau(X_2,\E_2,H_2)^{\otimes\chi(X_1,\E_1)}.
$$
\end{proposition}

\begin{proof}
For $k=0,1$, the space $\Omega^{\bar k}(X_1\times X_2,\E_1\boxtimes\E_2)$
has a dense subspace that is isomorphic to $\bigoplus_{l=0,1}
\Omega^{\bar l}(X_1,\E_1)\otimes\Omega^{\overline{k-l}}(X_2,\E_2)$.
Under this identification, the operators 
$d^{\E_1\boxtimes\E_2}=d^{\E_1}\otimes1+1\otimes d^{\E_2}$,
$(d^{\E_1\boxtimes\E_2})^\dagger=
(d^{\E_1})^\dagger\otimes1+1\otimes(d^{\E_2})^\dagger$,
$\Delta^{\E_1\boxtimes\E_2}=\Delta^{\E_1}\otimes1+1\otimes\Delta^{\E_2}$
on $\Omega^\bullet(X_1\times X_2,\E_1\boxtimes\E_2)$.
(We omit the superscript $H$ of the operators.)
Therefore an eigenvalue of $\Delta^{\E_1\boxtimes\E_2}$ is the sum of those
of $\Delta^{\E_1}$ and $\Delta^{\E_2}$ and the corresponding eigenspace is
a tensor product of those of the latters.
(All the eigenspaces are finite dimensional by the compactness assumption.)
In particular, $\ker(\Delta^{\E_1\boxtimes\E_2})=
\ker(\Delta^{\E_1})\otimes\ker(\Delta^{\E_2})$;
this is another proof of the K\"unneth isomorphism (Proposition~\ref{prop:ku})
when $X_1,X_2$ are both compact.
In addition, we have
$$
\spec'(\Delta_{\bar 0}^{\E_1\boxtimes\E_2})=
\{\lambda_1+\lambda_2>0\,|\,\lambda_i\in\spec(\Delta_{\bar 0}^{\E_i})
\text{ or }\lambda_i\in\spec(\Delta_{\bar 1}^{\E_i}),\;i=1,2\}
$$
and therefore by (\ref{part-zeta}), we get
\begin{align}\label{sum}
&\sum_{k=0,1}(-1)^k\zeta\big(s,(d_{\bar k}^{\E_1\boxtimes\E_2})^\dagger
   d_{\bar k}^{\E_1\boxtimes\E_2}\big) \\
=& {\small\mbox{$\displaystyle\sum_{k=0,1}\left(
 \mathop{\sum_{\lambda_i\in\spec(\Delta_{\bar k}^{\E_i}),i=1,2}}
 _{\lambda_1+\lambda_2\in\spec\I(\Delta_{\bar 0}^{\E_1\boxtimes\E_2})} 
 \!\!\!\!\!\frac{m\I(\lambda_1+\lambda_2,\Delta_{\bar 0}^{\E_1\boxtimes\E_2})}
 {(\lambda_1+\lambda_2)^s}\,
 -\!\!\!\!\!\mathop{\sum_{\lambda_i\in\spec(\Delta_{\bar k}^{\E_i}),i=1,2}}
 _{\lambda_1+\lambda_2\in\spec\II(\Delta_{\bar 0}^{\E_1\boxtimes\E_2})}
 \!\!\!\!\!\frac{m\II(\lambda_1+\lambda_2,\Delta_{\bar 0}^{\E_1\boxtimes\E_2})}
 {(\lambda_1+\lambda_2)^s}\right)$}}.               \nonumber
\end{align}

We now show that total contribution to the sum (\ref{sum}) from
$(\lambda_1,\lambda_2)$ vanishes if both $\lambda_1,\lambda_2>0$.
We consider four cases.\\
1. If both $\lambda_i\in\spec\I(\Delta_{\bar 0}^{\E_i})$, then there is
a non-zero $\omega_i\in\im(d_{\bar 0}^{\E_i})^\dagger$ such that
$\Delta_{\bar 0}^{\E_i}\omega_i=\lambda_i\omega_i$ for each $i=1,2$.
It is easy to see that 
$\omega_1\otimes\omega_2\in\im(d_{\bar 0}^{\E_1\boxtimes\E_2})^\dagger$
is an eigenvector of $\Delta_{\bar 0}^{\E_1\boxtimes\E_2}$ with eigenvalue
$\lambda_1+\lambda_2\in\spec\I(\Delta_{\bar 0}^{\E_1\boxtimes\E_2})$.
On the other hand, $d_{\bar 0}^{\E_i}\omega_i$ is a (non-zero) 
eigenvector of $\Delta_{\bar i}^{\E_i}$ with eigenvalue $\lambda_i$.
Hence $\lambda_i\in\spec'(\Delta_{\bar 1}^{\E_i})$.
Since $d_{\bar 0}^{\E_1}\omega_1\otimes d_{\bar 0}^{\E_2}\omega_2\in
\im(d_{\bar1}^{\E_1\boxtimes\E_2})$, we also have
$\lambda_1+\lambda_2\in\spec\II(\Delta_{\bar 0}^{\E_1\boxtimes\E_2})$.
So the contribution of $(\lambda_1,\lambda_2)\in
\spec'(\Delta_{\bar 0}^{\E_1})\times\spec'(\Delta_{\bar 0}^{\E_2})$
cancels that of $(\lambda_1,\lambda_2)\in
\spec'(\Delta_{\bar 1}^{\E_1})\times\spec'(\Delta_{\bar 1}^{\E_2})$.\\
2. Similarly, if both $\lambda_i\in\spec\II(\Delta_{\bar 0}^{\E_i})$
($i=1,2$), the corresponding contribution is also canceled.\\
3. If $\lambda_1\in\spec\I(\Delta_{\bar 0}^{\E_1})$ but
$\lambda_2\in\spec\II(\Delta_{\bar 0}^{\E_2})$, let $\omega_i$ ($i=1,2$)
be the corresponding eigenvectors of $\Delta_{\bar 0}^{\E_i}$.
Then $\omega_1\otimes\omega_2$ and 
$d_{\bar 0}^{\E_1}\omega_1\otimes(d_{\bar 1}^{\E_2})^\dagger\omega_2$ are
linearly independent eigenvectors of $\Delta_{\bar 0}^{\E_1\boxtimes\E_2}$
with the same eigenvalue $\lambda_1+\lambda_2$.
It is easy to see that one linear combination 
$\omega_1\otimes\omega_2-\lambda_1^{-1}d_{\bar 0}^{\E_1}\omega_1
\otimes(d_{\bar 1}^{\E_2})^\dagger\omega_2$
is in $\,\im(d_{\bar0}^{\E_1\boxtimes\E_2})^\dagger$, yielding
$\lambda_1+\lambda_2\in\spec\I(\Delta_{\bar 0}^{\E_1\boxtimes\E_2})$
while another (independent) combination $\omega_1\otimes\omega_2+\lambda_2^{-1}
d_{\bar 0}^{\E_1}\omega_1\otimes(d_{\bar 1}^{\E_2})^\dagger\omega_2$
is in $\,\im(d_{\bar1}^{\E_1\boxtimes\E_2})$, yielding
$\lambda_1+\lambda_2\in\spec\II(\Delta_{\bar 0}^{\E_1\boxtimes\E_2})$.
So the contributions of $(\lambda_1,\lambda_2)$ also cancel in this case.\\
4. The case $\lambda_1\in\spec\II(\Delta_{\bar 0}^{\E_1})$,
$\lambda_2\in\spec\I(\Delta_{\bar 0}^{\E_2})$ is similar. 

The non-zero contributions to (\ref{sum}) are thus from the subspaces
$\im(d_{\bar k}^{\E_1})^\dagger\otimes\ker(\Delta_{\bar l}^{\E_2})$,
$\ker(\Delta_{\bar l}^{\E_1})\otimes\im(d_{\bar k}^{\E_2})^\dagger\subset
\im(d_{\overline{k+l}}^{\E_1\boxtimes\E_2})^\dagger$, $k,l=0,1$.
Since $\dim\ker(\Delta_{\bar l}^{\E_i})=b_{\bar l}(X_i,\E_i,H_i)$, we have
\begin{align}
&\sum_{k=0,1}(-1)^k\zeta\big(s,(d_{\bar k}^{\E_1\boxtimes\E_2})^\dagger
   d_{\bar k}^{\E_1\boxtimes\E_2}\big)                          \nno
=&\sum_{k,l=0,1}(-1)^{k+l}\left(
  \sum_{\lambda_1\in\spec'((d_{\bar k}^{\E_1})^\dagger d_{\bar k}^{\E_1})}
  \!\!\frac{m\big(\lambda_1,(d_{\bar k}^{\E_1})^\dagger d_{\bar k}^{\E_1}\big)
  \,b_{\bar l}(X_2,\E_2,H_2)}{\lambda_1^s}\right.                \nno
 &\left.\quad\quad\quad\quad\quad\quad\quad\quad\quad
  +\sum_{\lambda_2\in\spec'((d_{\bar k}^{\E_2})^\dagger d_{\bar k}^{\E_2})}
  \!\!\frac{m\big(\lambda_2,(d_{\bar k}^{\E_2})^\dagger d_{\bar k}^{\E_2}\big)
  \,b_{\bar l}(X_1,\E_1,H_1)}{\lambda_2^s}\right)                \nno
&=\chi(X_2,\E_2)\sum_{k=0,1}(-1)^k
  \zeta\big(s,(d_{\bar k}^{\E_1})^\dagger d_{\bar k}^{\E_1}\big)
  +\chi(X_1,\E_1)\sum_{k=0,1}(-1)^k
  \zeta\big(s,(d_{\bar k}^{\E_2})^\dagger d_{\bar k}^{\E_2}\big)
  \nonumber
\end{align}
and therefore
$$
\frac{\Det'(d_{\bar 0}^{\E_1\boxtimes\E_2})^\dagger 
 d_{\bar 0}^{\E_1\boxtimes\E_2}}
{\Det'(d_{\bar 1}^{\E_1\boxtimes\E_2})^\dagger 
 d_{\bar 1}^{\E_1\boxtimes\E_2}}
=\left(\frac{\Det'(d_{\bar0}^{\E_1})^\dagger d_{\bar 0}^{\E_1}}
{\Det'(d_{\bar 1}^{\E_1})^\dagger d_{\bar1}^{\E_1}}\right)^{\chi(X_2,\E_2)}
\left(\frac{\Det'(d_{\bar0}^{\E_2})^\dagger d_{\bar 0}^{\E_2}}
{\Det'(d_{\bar 1}^{\E_2})^\dagger d_{\bar1}^{\E_2}}\right)^{\chi(X_1,\E_1)}.
$$
For the volume elements, we can choose
$$
\eta_{\bar k}^{\E_1\boxtimes\E_2}=\bigotimes_{l=0,1}
(\eta_{\bar l}^{\E_1})^{\otimes b_{\overline{k-l}}(X_2,\E_2,H_2)}\otimes
(\eta_{\bar l}^{\E_2})^{\otimes b_{\overline{k-l}}(X_1,\E_1,H_1)}
$$
and hence
$$
\eta_{\bar 0}^{\E_1\boxtimes\E_2}\otimes
(\eta_{\bar 1}^{\E_1\boxtimes\E_2})^{-1}
=(\eta_{\bar 0}^{\E_1}\otimes
(\eta_{\bar 1}^{\E_1})^{-1})^{\otimes\chi(X_2,\E_2)}
\otimes(\eta_{\bar 0}^{\E_2}\otimes
(\eta_{\bar 1}^{\E_2})^{-1})^{\otimes\chi(X_1,\E_1)}.
$$
\end{proof}

It would be interesting to establish the behavior of the analytic torsion
(form) under a general smooth fibration, analogous to \cite{BLo,DM,LST},
for the twisted de Rham or other $\ZA_2$-graded complexes.

\section{Calculations of analytic torsion and the simplicial analogue}
\label{sect:calc} 

\subsection{Analytic torsion when the flux is a top-degree form}\label{sect:n}
Recall that the twisted cohomology groups can be computed by the spectral
sequence in \S\ref{sect:twistedDR}.
In this process, each complex $(E_r^\bullet,\delta_r)$ is finite dimensional 
for $r\ge 2$ when $X$ is compact. 
The Knudsen-Mumford isomorphisms \cite{KM} 
$\det E_r^\bullet\cong\det E_{r+1}^\bullet$ for $r\ge2$ yield an isomorphism
\begin{equation}\label{km-isom}
\kappa\colon\det H^\bullet(X,\E)\to\det H^\bullet(X,\E,H) 
\end{equation}
since the spectral sequence converges to the twisted cohomology.  

\begin{proposition}\label{prop:3mfd}
Suppose $X$ is a compact oriented manifold of odd dimension and
$\E$ is a flat vector bundle associated to an orthogonal or unitary
representation of $\pi_1(X)$.
Assume $n=\dim X>1$ and let $H$ be an $n$-form on $X$.
Then
$$
\tau(X,\E,H)=\kappa(\tau(X,\E)).
$$
\end{proposition}

\begin{proof}
By Theorem~\ref{thm:indept}, we can choose a Riemannian metric on $X$ 
so that $\vol(X)=1$; let $\nu=*\,1$ be the volume form on $X$.
By Theorem~\ref{thm:indept-H}, we can also assume that $H=[H]\,\nu$,
where $[H]\in H^n(X,\RE)\cong\RE$ is a real number.
If $[H]=0$, then the statement is trivial; we assume that $[H]\ne0$.
Since $\E$ is a flat vector bundle associated to an orthogonal or unitary
representation, we have $H^0(X,\E)\cong\overline{H^n(X,\E)}^*$;
let $b_0:=\dim H^0(X,\E)=\dim H^n(X,\E)$. 
Let $\eta_i$ be the unit volume element of $H^i(X,\E)$ for $0\le i\le n$.
The metric-independent isomorphism (\ref{km-isom}) is given by
$$
\kappa\colon\bigotimes_{i=0}^n\eta_i^{(-1)^i}\mapsto
|[H]|^{b_0}\,\eta_{\bar0}\otimes\eta_{\bar1}^{-1}.
$$
Let $d_i$ ($0\le i\le n-1$) be the differential on $C^i=\Omega^i(X,\E)$.
Then $d_{\bar k}$ is equal to ${d_0\quad\;\,0\;\choose H\;\;d_{n-1}}$ on
$C^0\oplus C^{n-1}$, $d_i$ on $C^i$ for $i\le i\le n-2$, and $0$ on $C^n$.
Here $H$ also stands for taking wedge product with $H$.
The Ray-Singer torsion is 
$$
\tau(X,\E)=\prod_{i=0}^{n-1}(\Det'd_i^\dagger d_i)^{(-1)^i/2}
\bigotimes_{i=0}^n\eta_i^{(-1)^i}
$$
while the torsion for the twisted de Rham complex is
$$
\tau(X,\E,H)=\Det'{d_0^\dagger d_0+H^\dagger H\quad H^\dagger d_{n-1}
\choose\quad\;d_{n-1}^\dagger H\quad\;d_{n-1}^\dagger d_{n-1}}^{1/2}
\,\prod_{i=1}^{n-2}(\Det'd_i^\dagger d_i)^{(-1)^i/2}
\;\eta_{\bar0}\otimes\eta_{\bar1}^{-1}.
$$
The result follows from the following Lemma.
\end{proof}

\begin{lemma}\label{lem:KV}
Under the above assumptions, we have
$$
\Det'{d_0^\dagger d_0+H^\dagger H\quad H^\dagger d_{n-1}\choose
\quad\;d_{n-1}^\dagger H\quad\;d_{n-1}^\dagger d_{n-1}}
=[H]^{2b_0}\,\Det'd_0^\dagger d_0\,\Det'd_{n-1}^\dagger d_{n-1}.
$$
\end{lemma}

\begin{proof}
Let $Q_i$ be the orthogonal projection from (the completion of) $C^i$
onto $\ker(\Delta_i)$, $0\le i\le n$.
Set $\tilde\Delta_{n-1}=\Delta_{n-1}+Q_{n-1}$.
Then 
$$
\Det'{d_0^\dagger d_0+H^\dagger H\quad H^\dagger d_{n-1}\choose
\quad\;d_{n-1}^\dagger H\quad\;d_{n-1}^\dagger d_{n-1}}
=\Det'{\Delta_0+[H]^2\quad H^\dagger d_{n-1}\choose
\;\;d_{n-1}^\dagger H\quad\quad\tilde\Delta_{n-1}\;\;}
\,(\Det'd_{n-2}^\dagger d_{n-2})^{-1}.
$$
Since $d_{n-1}\tilde\Delta_{n-1}^{-1}d_{n-1}^\dagger=1-Q_n$ and hence 
$H^\dagger d_{n-1}\tilde\Delta_{n-1}^{-1}d_{n-1}^\dagger H=[H]^2(1-Q_0)$,
we have
\begin{align}
{\Delta_0+[H]^2\quad H^\dagger d_{n-1}\choose
\;\;d_{n-1}^\dagger H\quad\quad\tilde\Delta_{n-1}\;\;}
&={\Delta_0+[H]^2Q_0\quad H^\dagger d_{n-1}
\choose\quad\;\;0\quad\quad\quad\quad\tilde\Delta_{n-1}}
{\quad\quad 1\quad\quad\quad\;\;0
\choose\tilde\Delta_{n-1}^{-1}d_{n-1}^\dagger H\quad 1}      \nno
&={1\quad H^\dagger d_{n-1}\tilde\Delta_{n-1}^{-1}\choose
0\quad\quad\quad\;\;1\quad\quad\;\,}
{\Delta_0+[H]^2Q_0\quad0\;\;\choose
\quad\quad0\quad\quad\quad\tilde\Delta_{n-1}}
{\quad\quad 1\quad\quad\quad\;\;0
\choose\tilde\Delta_{n-1}^{-1}d_{n-1}^\dagger H\quad 1}.     \nonumber
\end{align}
We note that $n=\dim X$ is odd.
Since the elliptic pseudo-differential operators 
$$
{\Delta_0+[H]^2\quad H^\dagger d_{n-1}\choose
\;\;d_{n-1}^\dagger H\quad\quad\tilde\Delta_{n-1}\;\;},\quad
{\Delta_0+[H]^2Q_0\quad H^\dagger d_{n-1}
\choose\quad\;\;0\quad\quad\quad\quad\tilde\Delta_{n-1}},\quad
{\Delta_0+[H]^2Q_0\quad0\;\;\choose
\quad\quad0\quad\quad\quad\tilde\Delta_{n-1}}
$$
of order $2$ on $C^0\oplus C^{n-1}$ are odd in the sense of Kontsevich and
Vishik \cite{KV} and are invertible, the pseudo-differential operators
$$
{1\quad H^\dagger d_{n-1}\tilde\Delta_{n-1}^{-1}\choose
0\quad\quad\quad\;\;1\quad\quad\;\,},\quad
{\quad\quad 1\quad\quad\quad\;\;0
\choose\tilde\Delta_{n-1}^{-1}d_{n-1}^\dagger H\quad 1}
$$
of order $0$ are also odd and their determinants are defined \cite{KV};
we will denote these determinants by $\det$ to distinguish them from
zeta-function regularized determinants $\Det'$.
In fact, for any $a>0$, 
$$
\Det'{\Delta_0+a\quad H^\dagger d_{n-1}\choose
\quad0\quad\quad\;\;\tilde\Delta_{n-1}\;}
=\det{1\quad H^\dagger d_{n-1}\tilde\Delta_{n-1}^{-1}\choose
0\quad\quad\quad\;\;1\quad\quad\;\,}\;
\Det'{\Delta_0+a\quad\;\;0\;\;\choose\quad\;\;0\quad\;\;\tilde\Delta_{n-1}}.
$$
Choosing $a$ such that the spectrums of $\Delta_0+a$ and $\tilde\Delta_{n-1}$
are disjoint, the operators
$$
{\Delta_0+a\quad H^\dagger d_{n-1}\choose
\quad0\quad\quad\;\;\tilde\Delta_{n-1}\;},\quad
{\Delta_0+a\quad\;\;0\;\;\choose\quad\;\;0\quad\;\;\tilde\Delta_{n-1}}
$$
have identical spectrums and hence the same zeta-function regularized 
determinant. 
Thus
$$
\det{1\quad H^\dagger d_{n-1}\tilde\Delta_{n-1}^{-1}\choose
0\quad\quad\quad\;\;1\quad\quad\;\,}=1
$$
and, similarly,
$$
\det{\quad\quad 1\quad\quad\quad\;\;0
\choose\tilde\Delta_{n-1}^{-1}d_{n-1}^\dagger H\quad 1}=1.
$$
As determinants factorize for odd pseudo-differential operators of
non-negative order on an odd-dimensional manifold \cite{KV}, we get
\begin{align}
\Det'{\Delta_0+[H]^2\quad H^\dagger d_{n-1}\choose
\;\;d_{n-1}^\dagger H\quad\quad\tilde\Delta_{n-1}\;\;}
&=\Det'{\Delta_0+[H]^2Q_0\quad0\;\;\choose
\quad\quad0\quad\quad\quad\tilde\Delta_{n-1}}         \nno
&=[H]^{2b_0}\;\Det'd_0^\dagger d_0\;\Det'd_{n-1}^\dagger d_{n-1}
\;\Det'd_{n-2}^\dagger d_{n-2}
\nonumber
\end{align}
and the result follows.
\end{proof}

We note that neither Lemma~\ref{lem:KV} nor Proposition~\ref{prop:3mfd}
is valid if $\dim X=1$ and $[H]\ne0$.
We give a heuristic explanation of Lemma~\ref{lem:KV} when $n>1$.
For any $\lambda\in\spec'(d_0^\dagger d_0)$, let $\omega_\lambda$
be an eigenvector corresponding to $\lambda$.
Then $*d_0\omega_\lambda/\sqrt{\lambda}$ is an eigenvector of 
$d_{n-1}^\dagger d_{n-1}$ with the same eigenvalue.
On the subspace spanned by $\omega_\lambda$ and 
$*d_0\omega_\lambda/\sqrt{\lambda}$, the operator
$d_{\bar0}^\dagger d_{\bar0}$ acts as 
${\lambda+[H]^2\quad[H]\sqrt{\lambda}\choose [H]\sqrt{\lambda}
\quad\quad\lambda\quad}$, whose determinant is $\lambda^2$.
Notice that
$$
C^0\oplus\im(d_{n-1}^\dagger)=\ker(\Delta_0)\oplus\!\!\!\!
\bigoplus_{\lambda\in\spec'(d_0^\dagger d_0)}\!\!\!\!{\mathrm{span}}_\CO
\{\omega_\lambda,*d_0\omega_\lambda/\sqrt{\lambda}\}
$$
and $\ker(\Delta_0)$ is in the eigenspace of $\Delta_0+[H]^2Q_0$ corresponding
to the eigenvalue $[H]^2$ (with multiplicity $b_0$).
The ``product'' of these $\lambda^2$ together with $[H]^2$ leads to the result.

Under the assumptions of Proposition~\ref{prop:3mfd},
$\zeta(0,d_{\bar1}^\dagger d_{\bar1})$ for any $H$ is the same as its
value when $H=0$; it would be interesting to find the value of 
$\zeta(0,d_{\bar0}^\dagger d_{\bar0})$ when $[H]\ne0$.
(See Corollary~\ref{var-zeta0-H} and the discussion that follows.)

In addition to $\kappa$ in (\ref{km-isom}), there is another natural 
isomorphism $\kappa_0$ which maps between the alternating products
of unit volume elements, i.e.,
$$
\kappa_0\colon\bigotimes_{i=0}^n\eta_i^{(-1)^i}\mapsto
\eta_{\bar0}\otimes\eta_{\bar1}^{-1}.
$$
If $H$ is a top-degree form as in Proposition~\ref{prop:3mfd}, then
$\kappa_0$ is independent of the choice of metrics on $X$ and on $\E$. 
The appearance of $|[H]|$ in 
$$
\tau(X,\E,H)=|[H]|^{b_0}\,\kappa_0(\tau(X,\E))
$$
is consistent with the metric invariance of both $\tau(X,\E)$ and 
$\tau(X,\E,H)$ and dependence of the latter on the cohomology class $[H]$
only.


Proposition~\ref{prop:3mfd} applies especially to $3$-dimensional 
manifolds because $H$ is automatically a top-degree form if it contains
no $1$-form (which can be absorbed in the flat connection).
The Ray-Singer torsion has been calculated explicitly, directly or with 
the help of the Cheeger-M\"uller theorem, for many $3$-manifolds including
lens spaces \cite{Ray,Fr} and compact hyperbolic manifolds \cite{Fried}.
As a consequence, we get many non-trivial examples of analytic torsion 
for the twisted de Rham complexes of $3$-manifolds.

\subsection{Simplicial analogue of the torsion in a special case}\label{comb}
One of the standard ways to compute the classical Ray-Singer torsion is to
use the Cheeger-M\"uller theorem, relating it to the Reidemeister torsion.
Although there is difficulty in defining the simplicial counterpart
of the twisted analytic torsion in general, we will be able to do so under
the condition that the degree of the flux form is sufficiently high.

We first recall the construction of the Reidemeister torsion (cf.~\cite{M93}).
Suppose the manifold $X$ is equipped with a smooth triangularization or 
a CW complex structure.
Let $(C_\bullet(K),\partial)$ be the chain complex of the simplicial or
cellular complex $K$ with real coefficients.
Choose an embedding of $K$ as a fundamental domain in the corresponding
complex $\tilde K$ of the universal covering space $\tilde X$.
Then each $C_i(\tilde K)$ ($0\le i\le n$, where $n=\dim X$) is a free module
over the group algebra $\RE[\pi_1(X)]$ and the $i$-cells of $K$ form a basis.
Given a finite dimensional representation $\rho\colon\pi_1(X)\to\GL(E)$,
we define a cochain complex 
$$
C^\bullet(K,E):=\Hom_{\RE[\pi_1(X)]}(C_\bullet(\tilde K),E)
$$
with coboundary map $\partial^*$, whose cohomology is denoted by 
$H^\bullet(K,E)$.
With an Hermitian form on $E$, we choose a unit volume element of $E$.
This, together with the basis dual to the $i$-cells in $K$, defines a
volume element $\mu_i\in\det C^i(K,E)$.
We assume that the representation $\rho$ is unimodular.
Unimodularity means that $|\det\rho(\gamma)|=1$ for all $\gamma\in\pi_1(X)$.
Then the volume element $\mu_i$ is, up to a phase, independent of the choice
of the embedding of $K$ in $\tilde K$.
The {\em Reidemeister torsion} or {\em $R$-torsion}
$\tau(K,E)\in\det H^\bullet(K,E)$ is defined as the image
of $\otimes_{i=0}^n\,\mu_i^{(-1)^i}$ under the isomorphism
$\det C^\bullet(K,E)\cong\det H^\bullet(K,E)$.
It is invariant under subdivisions of the complex $K$.
If $X$ is odd-dimensional, the Euler number $\chi(K)=0$, and $\tau(K,E)$
(up to a phase) does not depend on the choice of the Hermitian form on $E$.
By the de Rham theorem, $H^\bullet(X,\E)\cong H^\bullet(K,E)$ and hence
$\det H^\bullet(X,\E)\cong\det H^\bullet(K,E)$.
The theorem of Cheeger and M\"uller \cite{C79,M78,M93} states that 
$\tau(X,\E)=\tau(K,E)$ under the above identification of determinant lines.

Recall that the cup product at the cochain level is associative but not
graded commutative.
We now assume that each homogeneous component of $H$ is of degree
greater than $\dim X/2=n/2$.
Let $h\in C^{\bar1}(K,E)$ be a representative of 
$[H]\in H^{\bar1}(X,\E)\cong H^{\bar1}(K,E)$.
Then since $h\cup h=0$, we have a $\ZA_2$-graded cochain 
complex $(C^\bullet(K,E),\partial^*_h)$, where
$\partial^*_h=\partial^*+h\cup\;\cdot\;$.
Denote its cohomology groups by $H^{\bar k}(K,E,h)$, $k=0,1$.
There is then an isomorphism 
$\det C^\bullet(K,E)\cong\det H^\bullet(K,E,h)$.
We define the twisted version of the $R$-torsion $\tau(K,E,h)$ as the image
of $\otimes_{i=0}^n\,\mu_i^{(-1)^i}$ under the above isomorphism.
This will be the simplicial counterpart of the analytic torsion $\tau(X,\E,H)$.

\begin{lemma}\label{lem:simp}
There is a canonical isomorphism $H^\bullet(X,\E,H)\cong H^\bullet(K,E,h)$.
\end{lemma}

\begin{proof}
Just as $\Omega^\bullet(X,\E)$, the complex $C^\bullet(K,E)$ has a filtration
$$
F^pC^{\bar k}(K,E)=\mathop{\bigoplus_{i\ge p}}_{i=k\!\!\!\!\mod 2}C^i(K,E),
$$
which yields a spectral sequence $\{{}'\!E_r^{pq},\delta'_r\}$ converging
to $H^\bullet(K,E,h)$.
The cochain map $\Omega^\bullet(X,\E)\to C^\bullet(K,E)$ that induces the
de Rham isomorphism preserves the filtrations.
Therefore there is a morphism of the the spectral sequences 
$\{E_r^{pq},\delta_r\}\to\{{}'\!E_r^{pq},\delta'_r\}$.
By the de Rham theorem, this morphism is an isomorphism starting with the
$E_2$-terms, which implies the result.
\end{proof}

We have the following analogue of the Cheeger-M\"uller theorem when
$H$ or $h$ is of top degree.

\begin{theorem}\label{thm:tcm}
With the same assumptions of Proposition~\ref{prop:3mfd} and under 
identification given by Lemma~\ref{lem:simp}, we have
$$
\tau(X,\E,H)=\tau(K,E,h).
$$
\end{theorem}

\begin{proof}
Let 
$$
\kappa'\colon\det H^\bullet(K,E)\to\det H^\bullet(K,E,h)
$$
be the isomorphism induced by the Knudsen-Mumford isomorphisms in the 
spectral sequence $\{{}'\!E_r^{pq}\}$.
The morphism of the two spectral sequences in the proof of 
Lemma~\ref{lem:simp} induces a commutative diagram
$$
\begin{CD}
\det H^\bullet(X,\E) @> \kappa >> \,\det H^\bullet(X,\E,H) \\
@V \cong VV @V \cong VV \\
\det H^\bullet(K,E) @> \kappa' >> \,\det H^\bullet(K,E,h).
\end{CD}
$$
By Proposition~\ref{prop:3mfd}, we have $\tau(X,\E,H)=\kappa(\tau(X,\E))$.
On the other hand, it is clear from the definition of $\tau(K,E,h)$ that
$\tau(K,E,h)=\kappa'(\tau(K,E))$.
The results follows from the Cheeger-M\"uller theorem $\tau(X,\E)=\tau(K,E)$
since the representation is orthogonal or unitary.
\end{proof}

Consider for example the lens space $X=L(1,p)$, $p\in\ZA$.
It has a cellular structure $K$ with one $i$-cell $e_i$ for each
$i=0,1,2,3$.
On the dual basis $e_i^*$ ($0\le i\le 3$), we have
$$\partial^*e_0^*=0,\quad\partial^*e_1^*=p\,e_2^*,\quad\partial^*e_2^*=0,\quad
\partial^*e_3^*=0.$$
So the Reidemeister torsion is $\tau(K)=|p|^{-1}\eta_0\otimes\eta_3^{-1}$.
If $h=q\,e_3^*$, then
$$\partial^*_he_0^*=q\,e_3^*,\quad\partial^*_he_1^*=p\,e_2^*,\quad
\partial^*_he_2^*=0,\quad\partial^*_he_3^*=0,$$
and the twisted torsion is $\tau(K,h)=|qp^{-1}|$.


\subsection{$T$-duality for circle bundles and analytic torsion}
Let $\TT$ be the circle group.
Suppose $X$ is a compact, oriented manifold and is the total space of 
a principal $\TT$-bundle
$$
\begin{CD}
\TT @>>> \,  X \\
&& @V \pi VV \\
&& M \end{CD}
$$
over a compact, oriented manifold $M$ and $H$, a closed $3$-form on $X$
that has integral periods.
The flat vector bundle $\E$ is taken to be the trivial real line bundle
with the trivial connection.
Let $\hat\TT$ be the dual circle group.
Then the $T$-dual principal circle bundle \cite{BEM}
$$ 
\begin{CD}
\hat \TT @>>> \hat X \\
&& @V\hat \pi VV     \\
&& M \end{CD}
$$
is determined topologically by its first Chern class $c_1(\hat X)=\pi_*[H]$.
We have the commutative diagram
$$
\xymatrix@=4pc@ur{X\ar[d]_{\pi} & X\times_M\hat X\ar[d]^{\hat p}\ar[l]_p\\
M & \hat X\ar[l]^{\hat\pi}}
$$
where $X\times_M\hat X$ denotes the correspondence space.
The Gysin sequence for $\hat X$ enables us to define a $T$-dual flux
$[\hat H]\in H^3(\hat X,\ZA)$ satisfying $c_1(X)=\hat\pi_*[\hat H]$
and $p^*[H]=\hat p^*[\hat H]\in H^3(X\times_M\hat X,\ZA)$.
Thus $T$-duality for circle bundles exchanges the $H$-flux on the one side
and the Chern class on the other.
It can be shown \cite{BEM} that 
$H^\bullet(X,H)\cong H^{\bullet+1}(\hat X,\hat H)$ and consequently,
\begin{equation}\label{id-line}
\det H^\bullet(X,H)\cong(\det H^\bullet(\hat X,\hat H))^{-1}.
\end{equation}
We wish to explore the relation between the twisted torsions
$\tau(X,H)\in\det H^\bullet(X,H)$ and 
$\tau(\hat X,\hat H)\in\det H^\bullet(\hat X,\hat H)$ under the above
identification.

We next explain $T$-duality at the level of differential forms.
Choosing connection 1-forms $A$ and $\hat A$ on the circle bundles $X$ and
$\hat X$, we define the metrics on $X$ and $\hat X$ by
$$  
g_X=\pi^*g_M+A\odot A,\quad g_{\hat X}=\hat\pi^*g_M+\hat A\odot\hat A,
$$
respectively.
Since a closed $3$-form is cohomologous to a $\TT$-invariant one and the
twisted cohomology groups depend only on the cohomology class of the flux
$H$ (Theorem~\ref{thm:indept-H}), we can assume, without loss of generality,
that $H$ is a $\TT$-invariant $3$-form on $X$.
Denote by $F,\hat F\in\Omega^2(M)$ the curvature $2$-forms of $A,\hat A$,
respectively.
Since $H-A\wedge\pi^*\hat F$ is a basic differential form on $X$, we have
$H=A\wedge\pi^*\hat F-\pi^*\Omega$ for some $\Omega\in\Omega^3(M)$.
Define the $T$-dual flux $\hat H$ by 
$\hat H=\hat\pi^*F\wedge\hat A-\hat\pi^*\Omega$.
Then $\hat H$ is closed and $\hat\TT$-invariant.
We define linear maps 
$T\colon\Omega^{\bar k}(X)\to\Omega^{\overline{k+1}}(\hat X)$ for $k=0,1$ by
$$
T(\omega)=\int_\TT e^{p^*\!A\wedge\hat p^*\!\hat A}\,p^*\omega,\quad
\omega\in\Omega^\bullet(X).
$$

\begin{lemma}\label{T-isom}
Under the above choices of Riemannian metrics and flux forms, 
$$
T\colon\Omega^{\bar k}(X)^\TT\to\Omega^{\overline{k+1}}(\hat X)^{\hat\TT},
$$
for $k=0,1$, are isometries, inducing isometries on the spaces of twisted
harmonic forms and hence on the twisted cohomology groups.
\end{lemma}

\begin{proof}
For any $\omega=\pi^*\omega_1+A\wedge\pi^*\omega_2\in\Omega^\bullet(X)^\TT$,
where $\omega_1,\omega_2\in\Omega^\bullet(M)$, we have
$T(\omega)=\hat\pi^*\omega_2+\hat A\wedge\hat\pi^*\omega_1$.
The isometry of $T$ follows from 
$$
\int_X\omega\wedge*_X\,\omega=
\int_M\omega_1\wedge*_M\,\omega_1+\int_M\omega_2\wedge*_M\,\omega_2.
$$
Since $d(p^*\!A\wedge\hat p^*\!\hat A)=-p^*H+\hat p^*\hat H$, we have
$T\circ d^H=d^{\hat H}\circ T$.
So $T$ acts on the spaces of twisted harmonic forms and on the twisted
cohomology groups.
\end{proof}

When $X$ is a $3$-manifold, Proposition~\ref{prop:3mfd} relates $\tau(X,H)$  
to $\tau(X)$, which can be calculated by the spectral sequence of fibration
\cite{DM,Fr,LST,Ma2}.

\begin{proposition}
Let $X$ be a oriented $3$-manifold which a $\TT$-fibration over a compact,
oriented surface $M$ and $H$, a flux $3$-form on $X$.
Suppose there is a T-dual fibration $\hat X$ with flux form $\hat H$.
Then under identification (\ref{id-line}), we have
\[  \frac{\tau(X,H)}{(2\pi)^{\chi(M)}}=
    \left[\frac{\tau(\hat X,\hat H)}{(2\pi)^{\chi(M)}}\right]^{-1}.   \]
\end{proposition}

\begin{proof}
We can choose the metrics and the flux forms on $X$, $\hat X$ as above.
Let $p=c_1(X)\in H^2(M,\ZA)\cong\ZA$ and $q=[H]\in H^3(X,\ZA)\cong\ZA$.
If $p=0$, then $X=M\times\TT$.
If $q=0$ as well, then
$\tau(X)=(2\pi)^{\chi(M)}\eta_{\bar0}^X\otimes(\eta_{\bar1}^X)^{-1}$.
If $q\ne0$, then by Proposition~\ref{prop:3mfd},
$$
\tau(X,H)=|[H]|\,\kappa_0(\tau(X))=(2\pi)^{\chi(M)}
|q|\,\eta_{\bar0}^{X,H}\otimes(\eta_{\bar1}^{X,H})^{-1}.
$$
If $p\ne0$ but $q=0$, then since the $\TT$-bundle $X\to M$ is oriented, we
can compute $\tau(X)$ by the Gysin sequence of the fibration $X\to M$
(see for example \cite{LST}, Corollary~0.9) and get
$\tau(X)=(2\pi)^{\chi(M)}|p|^{-1}\eta_{\bar0}^X\otimes(\eta_{\bar1}^X)^{-1}$.
If both $p,q\ne0$, then again by Proposition~\ref{prop:3mfd},
\begin{equation}\label{3-fib}
\tau(X,H)=|[H]|\,\kappa_0(\tau(X))=(2\pi)^{\chi(M)}
|qp^{-1}|\,\eta_{\bar0}^{X,H}\otimes(\eta_{\bar1}^{X,H})^{-1}.
\end{equation}
The result follows since $T$-duality interchanges $p$ and $q$ and since 
the isometries in Lemma~\ref{T-isom} identify $\eta_{\bar k}^{X,H}$ with 
$\eta_{\overline{k+1}}^{\hat X,\hat H}$ for $k=0,1$.
\end{proof}

We note that (\ref{3-fib}) is consistent with the simplicial calculation in 
\S\ref{comb} when $X=L(1,p)$, verifying Theorem~\ref{thm:tcm} in this case.
It can be generalized to the case when $X$ is an $S^k$-bundle over a
compact, oriented manifold $M$ of dimension $k+1$ and $H$ is a top
form on $X$.
The behavior of the twisted torsion under $T$-duality when $X$ is of any
dimension and $H$ is a closed $3$-form remains an interesting problem.
Such a relation will provide a new way of calculating twisted analytic torsions
and, in particular, the classical Ray-Singer torsion using $T$-duality.

\end{document}